\documentclass[letterpaper, 11pt,reqno]{amsart}
\usepackage{mathtools}
\usepackage{amsmath}
\usepackage{amssymb}
\usepackage{yhmath}
\usepackage{graphicx}
\usepackage{mathrsfs}
\usepackage{bbm}
\usepackage[dvipsnames]{xcolor}
\usepackage{tikz-cd}
\usepackage{tikz}
\usetikzlibrary{patterns}
\usepackage[hypertexnames=false]{hyperref}

\usepackage{verbatim}
\usepackage{float}

\DeclareMathAlphabet{\mathpzc}{OT1}{pzc}{m}{it}

\usepackage{thmtools}
\usepackage{thm-restate}

\usepackage{caption}

\usepackage[]{cleveref}
\newtheorem{theorem}{Theorem}[section]

\newtheorem*{claim*}{Claim}
\newtheorem{claim}{Claim}

\newtheorem{lemma}[theorem]{Lemma}
\newtheorem{lem}[theorem]{Lemma}
\newtheorem{corollary}[theorem]{Corollary}

\newtheorem{cor}[theorem]{Corollary}

\newtheorem{prop}[theorem]{Proposition}

\theoremstyle{definition}

\newtheorem{Def}[theorem]{Definition}

\theoremstyle{remark}
\newtheorem{remark}[theorem]{Remark}

\numberwithin{equation}{section}

\newcommand{\op}{\operatorname}

\newcommand{\be}{\begin{equation}}
\newcommand{\ee}{\end{equation}}
\newcommand{\Ga}{\Gamma}

\newcommand{\R}{\mathbb R}

\newcommand{\Z}{\mathbb Z}
\newcommand{\N}{\mathbb N}
\newcommand{\ga}{\gamma}

\newcommand{\la}{\lambda}
\newcommand{\La}{\Lambda}
\newcommand{\inte}{\op{int}}
\newcommand{\ba}{\backslash}

\newcommand{\cal}{\mathcal}
\newcommand{\br}{\mathbb R}

\newcommand{\PSL}{\op{PSL}}
\newcommand{\F}{\cal F}

\newcommand{\G}{\Gamma}

\renewcommand{\frak}{\mathfrak}

\newcommand{\e}{\varepsilon}

\renewcommand{\L}{\mathcal L}
\newcommand{\fa}{\mathfrak a}

\renewcommand{\i}{\op{i}}

\newcommand{\C}{\cal C}

\newcommand{\id}{\op{id}}

\renewcommand{\P}{\mathbb{P}}

\newcommand{\fg}{\frak g}

\renewcommand{\epsilon}{\e}

\newcommand{\SL}{\op{SL}}

\def\g{\gamma}

\newcommand{\fh}{\mathfrak{h}}
\renewcommand{\L}{\mathcal L}
\newcommand{\W}{\mathcal W}
\newcommand{\Hom}{\op{Hom}}
\renewcommand{\id}{\op{id}}
\newcommand{\Jor}{\op{Jor}}
\title[Limit cones and growth indicators]{Deformations of Anosov subgroups:\\ Limit cones and growth indicators}

\author{Subhadip Dey}
\address{Max Planck Institute for Mathematics in the Sciences, Leipzig,
    Germany\linebreak \indent
    {\rm\em Current address:} School of Mathematics, Tata Institute of Fundamental Research, Mumbai, India}
\email{subhadip@math.tifr.res.in}

\author{Hee Oh}
\address{Department of Mathematics, Yale University, New Haven, CT}
\email{hee.oh@yale.edu}

\begin{document}

\begin{abstract} 

Let $G$ be a connected semisimple real algebraic group. We prove that limit cones vary continuously under deformations of Anosov subgroups of 
$G$ under a certain convexity assumption, which turns out to be necessary. We apply this result to the notion of sharpness for the action of a discrete subgroup on a non-Riemannian homogeneous space. Finally, we show that, within the space of Anosov representations, the growth indicator, the critical exponents, and the Hausdorff dimension of limit sets (with respect to an appropriate non-Riemannian metric) all vary continuously.

\end{abstract}

\maketitle

\tableofcontents

\section{Introduction}
Let $G$ be a connected semisimple real algebraic group. 
The limit cone of a discrete subgroup of $G$ is
a fundamental object in the study of their asymptotic properties. 
 Let $\fa$ be a Cartan subalgebra of the Lie algebra $\fg$ of $G$, and 
let $\fa^+\subset \fa$ be a closed positive Weyl chamber.
Denote by $$\mu:G\to \fa^+$$
the Cartan projection (see \eqref{car}).
For any closed subgroup $\Ga$ of $G$, the limit cone $\L_\Ga$ is defined as the asymptotic cone of the Cartan projections of its elements:
$$\L_\Ga=\{\lim_{i\to \infty}  t_i \mu(\ga_i)\in \fa^+: t_i\to 0, \ga_i\in \Ga\} .$$
This notion was introduced by Benoist \cite{Be}.  
A natural question is whether the limit cones of discrete subgroups vary continuously in the deformation space. The main goal of this paper is
to answer this question affirmatively for Anosov subgroups, under a suitable convexity assumption.
This convexity assumption is, in fact, necessary. We  explain an application to the notion of sharpness introduced by Kassel and Kobayashi \cite{KK}. We also show that, along deformations
of Anosov subgroups, the following quantities vary continuously:
\begin{itemize}
    \item the growth indicator;
     \item the critical exponents;
    \item the Hausdorff dimension of the limit set, with respect to an appropriate non-Riemannian metric.
\end{itemize}

Anosov subgroups of $G$ are  regarded as the higher-rank analogues of convex cocompact subgroups in rank-one Lie groups. They play a central role in higher Teichm\"uller theory (\cite{La}, \cite{GW}). Let $\Pi$ be the set of all simple roots of $G$
with respect to $\fa^+$.  For any non-empty subset $\theta$ of $\Pi$,
 a finitely generated subgroup $\Ga<G$ is called $\theta$-Anosov if
there exists a constant $C>1$ such that, for all $\alpha\in \theta$ and $\ga\in \Ga$, 
$$\alpha(\mu(\ga))\ge C^{-1}|\ga| -C $$
where $|\ga |$ denotes the word  length of $\ga$  with respect to some fixed finite generating set of $\G$. When $\theta=\Pi$,
we speak of a Borel-Anosov subgroup or a $\Pi$-Anosov subgroup.
All Anosov subgroups in our paper are assumed to be non-elementary, i.e., not virtually cyclic.

\subsection*{Limit cones vary continuously} For a finitely generated subgroup $\Ga<G$, let $$\op{Hom}(\Ga, G)$$ be the space of all homomorphisms of $\Ga$ to $G$. Denote by $\mathsf C(\fa^+)$  the space of closed cones in the positive Weyl chamber $\fa^+$. Both spaces carry natural topologies (see \eqref{top} and \eqref{top2}).
Let $\id_\Ga\in \Hom(\Ga, G)$ be the inclusion map  from $\Ga$ to $G$.

Let  $\i:\Pi\to \Pi$ be the opposition involution of $G$ (see \eqref{i}).  For $\theta\subset \Pi$, define $\W_\theta$ to be the subgroup of the Weyl group consisting of those elements that fix, pointwise, the subspace $\fa_\theta =\bigcap_{\alpha\in \Pi-\theta}\op{ker}\alpha$.
A closed subgroup $\Ga<G$ is $\theta$-convex
if the orbit $\cal W_{\theta\cup\i(\theta)} \L_\Ga$ under the adjoint action is convex in $\fa$. 

\begin{theorem}\label{m1} Let $G$ be a connected semisimple real algebraic group.
Suppose that $\Ga<G$ is a $\theta$-Anosov and $\theta$-convex subgroup for some $\theta\subset \Pi$, and that its Zariski closure is reductive\footnote{A non-trivial algebraic subgroup of $G$ is reductive if its unipotent radical is trivial.}.
 Then  the map $$\sigma \mapsto
\L_{\sigma(\Ga)} \in \mathsf C(\fa^+)$$ is continuous at  $\id_\Ga \in \op{Hom}(\Ga, G)$.

Moreover, continuity can fail without $\theta$-convexity: there exists a $\theta$-Anosov subgroup $\Ga<G$ (for some $\theta\subset \Pi$) whose Zariski closure  is reductive but not $\theta$-convex, for which the map $\sigma\mapsto \L_{\sigma(\Ga)} $ is not continuous at $\id_\Ga$.
\end{theorem}

The continuity asserted in Theorem \ref{m1} splits into two parts:  lower semicontinuity (Proposition \ref{LC}) and upper semicontinuity (Theorem \ref{UC}). 
The reductive Zariski closure hypothesis (resp. the convexity hypothesis) is not required for the upper (resp. lower) semicontinuity. 
Many notable Anosov subgroups are known to have
reductive Zariski closure--for instance, maximal representations \cite{BIW} and Hitchin representations \cite{Sam}.

The failure of continuity in the absence of the convexity assumption is discussed in section \ref{fa}.

We show that each of the following classes satisfies the hypotheses of Theorem \ref{m1} (see section \ref{ex}):
\begin{corollary} \label{three0}
Suppose one of the following holds:
\begin{enumerate}
 \item $\op{rank} G=2$ and $\G$ is an Anosov subgroup of $ G $;
 \item $\Ga$ is a Borel-Anosov subgroup of $G$;
    \item $\G$ is a $ \theta $-Anosov subgroup of $ G $ such that $ \L_\Ga $ is a convex cone contained in $ \fa_{\theta \cup\i(\theta)}$ for some $\theta\subset \Pi$;
    \item $\G$ is a non-elementary convex cocompact subgroup of a rank-one simple algebraic subgroup of $ G $.
\end{enumerate}

If moreover $\Ga$ has reductive Zariski closure, then the map
$\sigma \mapsto \L_{\sigma(\Ga)}$
is continuous at $ \id_\Ga \in \Hom(\Ga, G) $.
\end{corollary}
 
For case (2), we do not need the reductive Zariski closure assumption (Proposition \ref{borel}).
This case can also be deduced from the work of Breuillard and Sert on continuity of joint spectrum \cite[Theorem 1.7]{BS}.
For a Zariski dense Borel-Anosov subgroup that is
isomorphic to the fundamental group of a closed negatively curved manifold, the same result was obtained by Sambarino via thermodynamic formalism \cite{Sam1}. 
Case (4)  was proved by Kassel \cite{Ka}.
Our proof of Theorem \ref{m1} is completely different from those in \cite{Ka}, \cite{Sam1} and \cite{BS}. It is based on the local-to-global principle of Morse quasi-geodesics due to Kapovich-Leeb-Porti \cite{KLP2}. In particular, we show that the limit cone of a small perturbation of a $\theta$-Anosov subgroup $\Ga$ is contained in any $\theta\cup\i(\theta)$-admissible cone whose interior contains $\L_\Ga-\{0\}$ (Proposition \ref{ucp}). The $\theta$-convexity assumption on $\Ga$ ensures that $\L_\Ga$ can be approximated by such admissible cones (Lemma \ref{con1}).

\medskip 
We now discuss some applications of the upper semicontinuity of limit cones (Theorem \ref{UC}).
\subsection*{Sharpness is an open condition} For a closed subgroup $H$ of $G$, a discrete subgroup $\Ga$ of $G$ is called {\it sharp} for $G/H$ if its limit cone meets that of $H$ only at zero:
$$\L_\Ga\cap \L_H=\{0\}.$$ Introduced by Kassel-Kobayashi \cite{KK}, sharpness plays a key role in the study of compact Clifford-Klein forms. If $\Ga$ is sharp for $G/H$, then it acts properly discontinuously on $G/H$ (\cite{Be1}, \cite{Ko}). Conversely, Kassel and Tholozan  recently proved that
if $\Ga$ acts properly discontinuously and cocompactly on $G/H$, then $\Ga$ is sharp for $G/H$ \cite{KT}.
 Theorem \ref{UC} immediately gives the 
following openness statement: 

\begin{cor}\label{sh}
 Let $H<G$ be a closed subgroup, and let
$\Ga<G$ be a $\theta$-Anosov and $\theta$-convex subgroup  for some $\theta\subset \Pi$. If $\Ga$ is sharp for $G/H$, then there exists an open neighborhood $\cal O$ of $\id_\Ga$ in $\Hom(\Ga, G)$
such that for all $\sigma\in \cal O$, the subgroup $\sigma(\Ga)$ is sharp for
$G/H$.
\end{cor}

When $H$ is reductive of co-rank one in $G$, this corollary
 was proved in \cite[Corollary 6.3]{KT} without the convexity assumption.
We refer to \Cref{sharp} (and Corollary \ref{sharp2}) for a more general statement which in particular yields another proof of  \cite[Corollary 6.3]{KT}.

\subsection*{Growth indicators vary continuously} Our proof of Theorem \ref{m1} naturally extends to the study of continuity of
$\theta$-limit cones (Theorem \ref{tl}).
Using this, we show that
$\theta$-growth indicators vary continuously in the deformation space of a $\theta$-Anosov subgroup.  This is a higher rank analog of the classical result that the Riemannian critical exponents vary continuously (even analytically) in the space of convex cocompact representations of a finitely generated group (\cite{Ru}, \cite{AR}).

In the rest of the introduction, we fix a non-empty subset $\theta\subset \Pi$. Let $p_\theta:\fa\to \fa_\theta $ be the canonical projection invariant under $\W_\theta$. Consider the $\theta$-Cartan projection $$\mu_\theta=p_\theta\circ\mu:G\to \fa_\theta^+$$ where  $\fa_\theta^+=\fa_\theta\cap \fa^+$. For a closed subgroup $\Ga<G$, its $\theta$-limit cone $\L_\Ga^\theta $ is defined as
the asymptotic cone of $\mu_\theta(\Ga)$.
 When $\Ga$ is $\theta$-discrete, that is, $\mu_\theta |_{\Ga}$ is proper,
 the growth indicator $\psi^\theta_\Ga:\fa_\theta^+\to \br\cup\{-\infty\}$ is well-defined (see  Definition \ref{growth}). It equals $-\infty$ outside $\L_\Ga^\theta$, and is
 non-negative on $\L_\Ga^\theta$: specifically, for $v\in \L^\theta_\Ga$,
\be\label{ttt0} \psi_\Ga^\theta(v)=\|v\|\inf_{v\in \cal C}\limsup_{T\to\infty} \frac{1}{T} {\log \#\{\ga\in \Ga: \mu_\theta(\ga)\in \cal C, \|\mu_\theta(\ga)\|\le T\}}\ee 
where the infimum runs over all open cones $\C\subset \fa_\theta^+$ containing $v$ and $\|\cdot\|$ is a norm on $\fa_\theta$.  Any discrete subgroup of $G$ is $\Pi$-discrete and the function $\psi_\Ga^\Pi$ was
 introduced by Quint \cite{Q2}. This notion was extended to general $\theta$ in \cite{KOW}. 
 The growth indicator $\psi_\Ga^\theta$ records the exponential growth rate of $\Ga$ in all directions of $\fa_\theta^+$, generalizing the classical Riemannian critical exponent in rank one. Denote by $\inte\L_\Ga^\theta$ the relative interior of $\L_\Ga^\theta$ in $\fa_\theta$. When $\Ga$ is Zariski dense in $G$,  the cone $\L_\Ga^\theta$ is convex with non-empty interior \cite{Be}. All $\theta$-Anosov subgroups are $\theta$-discrete.

\begin{theorem}\label{iuc} Let $\G$ be a Zariski dense $\theta$-Anosov subgroup of $G$.
If $\sigma_n\to \id_\Ga$ in $\Hom(\Ga, G)$, then 
$$\lim_{n\to \infty} \psi^\theta_{\sigma_n(\Ga)}(v) = \psi^\theta_\Ga (v) \quad\text{ for all $v\in \inte \L_\Ga^{\theta}$} $$
and the convergence is uniform on compact subsets of $\inte \L_{\Ga}^\theta$.
\end{theorem}

For a $\theta$-discrete subgroup $\Ga<G$, the following critical exponent is well-defined:
$$ \delta_{\Ga}^{\theta}\coloneqq\limsup_{T\to \infty}\frac{1}{T}{\log \#\{\ga\in \Ga: \|\mu_\theta(\ga)\|\le T\}}.
$$
If $\Ga$ is Zariski dense,
 then there exists a unique maximal growth direction $ u_\Ga^\theta\in \fa_\theta^+$ such that
$\delta_\Ga^\theta=\psi_\Ga^\theta(u_\Ga^\theta)=\max_{v\in \fa_\theta, \|v\|=1}\psi_\Ga^\theta(v) $ (\cite{Q2}, \cite{KOW}).
\begin{cor}\label{us}
     If $\G<G$ is a Zariski dense $\theta$-Anosov subgroup, then
 the maps $$\sigma\mapsto \delta_{\sigma(\Ga)}^{\theta}   \quad \text{ and }\quad \sigma \mapsto u_{\sigma(\Ga)}^\theta $$ are continuous at $\id_\Ga\in \Hom(\Ga, G)$.  \end{cor}

For a Borel-Anosov subgroup $\Ga<G$, isomorphic to the fundamental group of a closed negatively curved manifold, Theorem \ref{iuc} and Corollary \ref{us} were  proved by Sambarino \cite{Sam1}.  Unlike his proof,
our proof uses  conformal measure theory, inspired by the work of Sullivan \cite{Su} and McMullen \cite{Mc}.

\subsection*{Analyticity of Hausdorff dimension}
 Let $\fa_\theta^*$ be the space of linear forms on $\fa_\theta$;
equivalently, the space of linear forms on $\fa$ that are $p_\theta$-invariant. Let $\Ga$ be a $\theta$-Anosov subgroup and let $\La_\Ga^\theta$ be its $\theta$-limit set (see Definition \ref{tl}).
Given $\phi\in \fa_\theta^*$ with $\phi>0$ on $\L^\theta_\Ga-\{0\}$,
define the premetric $d_\phi(\xi, \eta)=e^{-\phi(\cal G(\xi,\eta))}$ for any
$\xi\ne \eta$ in $\La_\G^\theta$, where $\cal G$ denotes the $\fa$-valued Gromov product (see \cite[Section 5]{DKO}).
Let $\dim_\phi(\La_\G^\theta)$ denote the Hausdorff dimension of $\La_\G^\theta$ with respect to $d_\phi$. 
Using the work of Dey-Kim-Oh \cite[Corollary 9.9]{DKO} and  a strengthened version of the Bridgeman-Canary-Labourie-Sambarino theorem (Theorem \ref{bcls}), we obtain:

\begin{cor}
    \label{dko}
    Let $\Ga$ be a $\theta$-Anosov subgroup with $\L_\Ga^\theta$ convex (e.g., $\Ga$ is Zariski dense). Let $\phi \in \fa_{\theta}^*$ be positive on $\L_{\Ga}^{\theta}-\{0\}$.  If $D$ is an analytic family of $\theta$-Anosov representations in
    $\Hom(\Ga, G)$, then $$\sigma\mapsto \dim_\phi (\La^{\theta}_{\sigma(\Ga)} )$$ is analytic in $D$.
\end{cor}
When $\phi$ is non-negative on $\fa_\theta^+-\{0\}$, this result appears in \cite[Corollary 9.13]{DKO}.
We note that the analogous statement fails if we replace $\dim_\phi (\La^{\theta}_{\sigma(\Ga)} )$ by the Hausdorff dimension of $\La^{\theta}_{\sigma(\Ga)}$ with respect to a Riemannian metric (see \cite{LPX}).
\medskip 
\subsection*{More general critical exponents} For any discrete subgroup $\Ga<G$ and any linear form $\phi\in \fa^*$ that is positive on $\L_\Ga-\{0\}$, we can define
the $\phi$-critical exponent $0\le \delta_{\phi,\Ga} <\infty $ by
$$\delta_{\phi, \Ga}=\limsup_{T\to \infty}\frac{\log \#\{\ga\in \Ga: \phi(\mu(\ga))\le T\}}{T} .$$

\begin{theorem}\label{m2}
    Let $\Ga$ be a $\theta$-Anosov subgroup of $G$ such that $\L_\Ga $ is a convex cone contained
    in $\fa_\theta$. Let $\phi\in \fa^*$ satisfy $\phi>0$ on $\L_\Ga -\{0\}$. Then for all
    $\sigma  \in \op{Hom}(\Ga, G)$ sufficiently close to $\id_\Ga$,  
the critical exponent $0<\delta_{\phi, \sigma(\Ga)}<\infty$ is well-defined and
    the map 
    $$\sigma\mapsto 
    \delta_{\phi, \sigma(\Ga)}$$ is continuous at  $\id_\Ga \in \op{Hom}(\Ga, G)$.
\end{theorem}
The key point is that the linear form $\phi\in \fa^*$ need not be $p_\theta$-invariant.
\begin{cor} \label{coo}  Let $\Ga$ and $\phi$ be as in Theorem \ref{m2}.
  Suppose that for some sequence $\sigma_n \to \id_\Ga $ in $\Hom(\Ga, G)$ and some $\kappa>0$,
    $$\sup_n \psi_{\sigma_n(\Ga)}(v) \le  \kappa \phi(v)\quad\text{ for all $v\in \fa^+$.}$$
    Then
    $$\psi_\Ga \le \kappa \phi .$$
\end{cor}
Denote by $\rho$ the half sum of all positive roots of $(\fg, \fa)$ counted with multiplicity.
The inequality $\psi_\Ga \le  \rho$ is equivalent to  the quasi-regular representation
$L^2(\Ga\ba G)$ being tempered (\cite{EO}, \cite{LWW}).
Applied to the linear form $\rho$, Corollary \ref{coo} asserts that
if each $L^2(\sigma_n(\Ga)\ba G)$ is tempered, then $L^2(\Ga\ba G)$ is tempered as well.
This statement was proved earlier in \cite{FO} by studying the matrix coefficients, but Corollary \ref{coo} provides finer information about the behavior of growth indicators since it allows the use of more general linear forms.

\medskip

\subsubsection*{Acknowledgement}
We would like to thank Misha Kapovich, Fanny Kassel,  Dongryul Kim,  Eduardo Reyes,  Max Riestenberg, and Kostas Tsouvalas for useful discussions. We thank the anonymous referee for their careful reading and comments.

\section{Limit cones and  reductive Zariski closures}\label{lower}
Let $G$ be a connected linear reductive real algebraic group.
Let $A$ be a maximal real split torus of $G$.
Let $\fg$ and $\fa$  denote the Lie algebras of $G$
and $A$, respectively. Fix a positive Weyl chamber $\fa^+ \subset \fa$.
Fix a maximal compact subgroup $K< G$ such that the Cartan decomposition $G=K 
(\exp \fa^+) K$ holds. For $g\in G$,
there exists a unique element $\mu(g)=\mu_G(g)\in \fa^+$ such that $g\in K \exp (\mu(g)) K$, called the Cartan projection of $g$. The Cartan projection map 
\be\label{car} \mu=\mu_G:G\to\fa^+\ee  is continuous and proper.

Let $\Phi^+=\Phi^+(\fg, \fa)$ denote the set of all roots and $\Pi\subset \Phi^+$  the set of all  simple roots given by the choice of $\fa^+$.
We denote by $\langle\cdot,\cdot\rangle$ and $\|\cdot\|$ the inner product and norm on $\fg$ respectively, induced by the Killing form. 
We use the notation $\fa^1$ for the unit sphere in $\fa$.

Let $d$ denote the left $G$-invariant and right $K$-invariant distance on $G$ such that $d(e, g)=\|\mu(g)\|$ for all $g\in G$.

The Weyl group $\cal W$ is given by $N_K(A)/C_K(A)$, where
 $N_K(A)$ and $C_K(A)$ denote the normalizer and centralizer of $A$ in $K$, respectively. 
The Weyl group $\cal W$ acts on $\fa$ by the adjoint action: 
$$w.v=\op{Ad}_w(v)$$ for $w\in \cal W$ and $v\in \fa$.
 Let $\i:\fa \to \fa $ denote the opposition involution, that is,
\be\label{i} \i (v)=-\op{Ad}_{w_0}(v)\quad\text{for $v\in \fa$} \ee 
where $w_0$ is the longest Weyl element. It induces the involution $ \i:
\Pi\to \Pi$
which we also denote by $\i$: $\i(\alpha)=\alpha\circ \i$ for all $\alpha\in \Pi$.

	\subsection*{Limit cone}
 The \emph{limit cone} $\L_\Ga=\L_{\Ga, G}$ of a closed subgroup $\Gamma$ of $G$ is defined as the asymptotic cone of $\mu(\Ga)$:
 $$\L_\Ga=\{\lim t_i \mu(\ga_i)\in \fa^+: t_i\to 0, \ga_i\in \Ga\} .$$
Since $\mu(g^{-1}) = \i(\mu(g))$ for all $g \in G$, we have $\i(\L_\Ga)=\L_\Ga$. 

Any $g\in G$ can be written as the commuting product $g=g_hg_e g_u$ where $g_h$ is hyperbolic, $g_e$ is elliptic and $g_u$ is unipotent. 
	The hyperbolic component $g_h$ is conjugate to a unique element of the form $\exp(\lambda(g))$ where $ \lambda(g) \in \fa^+$. The element $\lambda(g)$ is called
	the Jordan projection of $g$. The Jordan projection $$\lambda:G\to \fa^+$$
is continuous. For a closed subgroup $\Gamma$ of $G$, we denote by 
$$\L_\Ga^{\Jor}=\L_{\Ga,G}^{\Jor}$$  the smallest closed cone of $\fa^+$ containing the  Jordan projection $\lambda(\Gamma)$.
Since $\la(g^n)=n\la(g)$ for all $g\in G$ and $n\in \N$, it follows that  $\L_\Ga^{\Jor}$ is equal to the asymptotic cone of $\lambda(\Ga)$.
Since for any $ g \in G $,
\[
\lambda(g) = \lim_{n \to \infty} \frac{\mu(g^n)}{n},
\]
we have the inclusion
\be\label{eqn:JLC}
\mathcal{L}_\Gamma \supset \mathcal{L}_\Gamma^{\rm Jor}.
\ee
 In general, $\cal L_\Ga\ne \L_\Ga^{\rm Jor}$. For example, when $\Ga$
is a unipotent subgroup, $\L_\Ga^{\rm Jor}=\{0\}$ while $\cal L_\Ga$ is non-trivial.

    \begin{theorem}\cite{Be} \label{be0} 
If  $\Ga$  is Zariski dense in $G$, then 
$\L_\Ga$ is convex and has non-empty interior.
Moreover, we have
    \be\label{jor0} \mathcal{L}_\Gamma = \mathcal{L}_\Gamma^{\Jor}. \ee 
\end{theorem}

\subsection*{Limit cone of discrete subgroups with reductive Zariski closure}
The equality \eqref{jor0} holds more generally when the Zariski closure of $\Ga$ is a reductive subgroup of $G$. To discuss this, let $H$ be a 
 reductive algebraic subgroup  of $G$ in this subsection. There exists a Cartan involution of $\fg$ whose restriction to $\fh=\op{Lie}H$ is also a Cartan involution of $\fh$ \cite{Mo}. It follows that,
by replacing $K$ and $A$ up to conjugation, we may assume that $K\cap H$ is a maximal compact subgroup of $H$ and $A\cap H$ is a maximal real split torus of $H$. We will denote by $\fa_H$ the Lie algebra of $A\cap H$.
Choosing a positive Weyl chamber $\fa_H^+$, we have
\be\label{ch} H=(K\cap H)(\exp \fa_H^+) (K\cap H)\ee  and the corresponding the Cartan projection map $\mu_H:H\to \fa_H^+$.  When $\fa_H^+$ is not contained in $\fa^+$, we have $\mu_H \ne \mu_G|_H$.  
For $w\in \W$ and $v\in \fa$, we simply write $wv=w.v=\text{Ad}_w(v)$, omitting the dot.
\begin{lemma} \label{hg} 
For any closed subgroup $\Ga$ contained in $H$, we have
 \be\label{f0} \L_{\Ga,G} =\bigcup_{w\in \W} \{ w^{-1} v : v\in \L_{\Ga, H} \cap w \fa^+\} .\ee 
In particular,
 \be\label{f1} \mu_G(\exp \L_{\Ga, H})=\L_{\Ga, G}.\ee 
 Moreover, the same statements hold for $\L_{\Ga, G}^{\Jor}$ and $\L_{\G, H}^{\Jor}$.
\end{lemma}
\begin{proof}
  Recall that $\fa=\bigcup_{w\in \W}  w (\fa^+)$ and if $v\in w_1\fa^+\cap w_2\fa^+$ for some $w_1,w_2\in \W$, then $w_1^{-1} v=w_2^{-1} v$. Moreover, $w ( \inte\fa^+)$, $w\in \W$, are disjoint from each other. 

Write $\fa_H^+=\bigcup_{w\in \W} (\fa_H^+\cap w\fa^+)$. Then $\mu_G$ on $\exp \fa_H^+$ is given as follows:  $$\mu_G(\exp v)= w^{-1} v\quad\quad \text{if $v\in \fa_H^+\cap w\fa^+$.}$$
Hence \eqref{f1} follows from the first claim.

  Suppose that  $v\in \L_{\Ga, H}\cap w\inte \fa^+$ for some $w\in \W$.
 By definition, for some $t_i>0$ and $\ga_i\in \G$ tending to $\infty$, we have
 $v=\lim t_i \mu_H(\ga_i)$.
    Then $\mu_H(\ga_i)\in w \inte \fa^+$ for all sufficiently large $i$ and hence 
    $w^{-1}\mu_H(\ga_i)=\mu_G(\ga_i)$.
    So
    $$w^{-1} v=\lim t_i w^{-1} \mu_H(\ga_i)=
    \lim t_i\mu_G(\ga_i).$$
    Therefore $w^{-1} (\L_{\Ga, H}\cap w\inte \fa^+)  \subset \L_{\Ga, G}$ for each $w\in \W$.
 By the continuity of the Cartan projection $\mu_G$, this implies the inclusion $\supset$ of \eqref{f0}. 
 
 Conversely, if $v\in \L_{\Ga, G}\cap \inte\fa^+$ is given by $\lim t_i \mu_G(\ga_i)$ for sequences  $t_i>0$ and
 $\ga_i\in \Ga$ tending to $\infty$, then $\mu_G(\ga_i)\in \inte\fa^+$ for all large $i$. Moreover, there exists $w\in \W$ such that $\mu_H(\ga_i)\in w\inte \fa^+$ for all sufficiently large $i$.
 Hence $\mu_G(\ga_i)=w^{-1} \mu_H(\ga_i)$ for all sufficiently large $i$.
Therefore
if we set  $v'=\lim t_i \mu_H(\ga_i)$, then $v'\in \L_{\Ga, H}\cap w\fa^+$ and
$v=  w^{-1} v'$. 
Hence $$\L_\Ga \cap \inte\fa^+\subset \bigcup_{w\in \W} \{ w^{-1} v : v\in \L_{\Ga, H} \cap w \fa^+\}.$$
 Again by continuity of $\mu_G$,  this implies the inclusion $\subset$ of \eqref{f0}.
 Since $\L_\Ga^{\Jor}$ is the asymptotic cone of $\lambda(\Ga)$, the same proof replacing the Cartan projection $\mu$ by the Jordan projection
 $\lambda$ proves the last claim about $\L_\Ga^{\Jor}$ and $\L_{\Ga, H}^{\Jor}$.
\end{proof}
   
\begin{theorem}\label{be} 
If $\Ga$ is a closed subgroup of $G$ whose 
Zariski closure is reductive, then
    \be\label{jor} \mathcal{L}_\Gamma = \mathcal{L}_\Gamma^{\Jor}. \ee 
\end{theorem}
	
\begin{proof} If  $H$ denotes the Zariski closure of $\Ga$, then,
 by Theorem \ref{be0}, we have
    $\L_{\Ga, H}=\L^{\Jor}_{\Ga, H}.$  Hence the conclusion follows from Lemma \ref{hg}. \end{proof}

In general, the reductive Zariski closure condition on $\Ga$ does not imply the convexity of the limit cone (see section \ref{fa}). On the other hand, we have the following:
\begin{lemma}\label{convex1} If $\Ga$ is a discrete subgroup of $G$ with reductive Zariski closure such that $\L_\Ga\subset \inte\fa^+\cup\{0\}$, then
$\L_\Ga$ is convex.
\end{lemma}
\begin{proof} Let $H$ be the Zariski closure of $\Ga$.
   By Lemma \ref{hg}., $\L_\Ga =\{ w^{-1} v : v\in \L_{\Ga, H} \cap w \fa^+\}$.
    By the hypothesis that $\L_\Ga \subset  \inte\fa^+\cup\{0\} $, the set
    $\L_{\Ga, H}-\{0\}$ is contained in the disjoint union $ \bigcup_{w\in \W} (\fa_H^+\cap w \inte \fa^+) $. As $\L_{\Ga, H}$ is convex by Theorem \ref{be0}, it follows that for some unique $w\in \W$,
    $$\L_{\Ga, H}\subset  (\fa_H^+\cap w \inte \fa^+). $$ Therefore,
    $\L_\Ga = w^{-1} \L_{\Ga, H}$, and hence it is convex.
\end{proof}

\subsection*{$\theta$-limit cones}
For a non-empty subset $\theta\subset \Pi$,  set 
\be\label{at0} \fa_\theta=\bigcap_{\alpha\in \Pi-\theta} \ker\alpha.\ee 
Denote by $\cal W_\theta$ the subgroup consisting  of all elements of the Weyl group $\cal W$ fixing $\fa_\theta$ pointwise; it is generated by reflections with respect to the walls $\ker \alpha$, $\alpha\in \Pi-\theta$.

Let \be\label{pt} p_\theta:\fa\to \fa_\theta\ee be  the unique projection that is invariant under $\W_\theta$.  Since $\W$ acts by isometries on $\fa$, $p_\theta$ can also be defined as the orthogonal projection to $\fa_\theta$. In particular, $p_\theta\vert_{\fa_\theta}$ is the identity map.
Set $\fa_\theta^+=\fa_\theta\cap \fa^+$.
Let $$\mu_\theta:=p_\theta\circ\mu:G\to \fa_\theta^+ .$$
For a closed subgroup $\Ga$ of $G$,
its $\theta$-limit cone $\L_\Ga^{\theta}$ is defined as the asymptotic cone of $\mu_\theta(\Ga)$:
$$\L_\Ga^{\theta}:=\{\lim t_i \mu_\theta(\ga_i)\in \fa_\theta^+: t_i\to 0, \ga_i\in \Ga\} .$$
We have $\L_\Ga^\theta =p_\theta(\L_\G).$
Note that $\fa_\Pi^+=\fa^+$ and $\L_{\Ga}^\Pi=\L_\Ga$.

\section{Lower semicontinuity of limit cones}\label{lower1}
Let $G$ be a connected semisimple real algebraic group. We keep the notation from section \ref{lower}.
\begin{Def}\ \label{top} 
 For a finitely generated group $\Ga$,
the topology on $\op{Hom}(\Ga, G)$ is given as follows: for a fixed finite generating subset $S$ of $\Ga$,  we have $\sigma_n\to \sigma $ in $\op{Hom}(\Ga, G)$ if for each $\ga\in S$, $\sigma_n(\ga)\to \sigma(\ga)$ as $n\to \infty$. 
This topology is  independent of the choice of a generating subset.
\end{Def}
If $\Ga<G$, we denote by $\id_\Ga$ the inclusion map, and
then the subsets
$\{\sigma\in \Hom(\Ga, G): \max_{\ga\in S} d(\sigma(\ga), \ga)< \e\}$, $\e>0$,
form a basis of neighborhoods of $\id_\G$.

Throughout the paper, all closed cones in $\fa_\theta^+$ are assumed to be non-degenerate, i.e., they are not $\{0\}$.

\begin{Def}\ \label{top2} For a subset $\theta\subset \Pi$,
 denote by $\mathsf C(\fa_\theta^+)$ the space of all closed cones in $\fa_\theta^+$  with the topology in which
 $\mathcal C_i\to \mathcal C$ if the Hausdorff distance between $\mathbb P(\mathcal C_i)$ and $\mathbb P(\mathcal C)$ in $\mathbb P(\fa_\theta^+)$ converges to zero as $i\to\infty$.
With this topology, ${\mathsf C}(\fa_\theta^+)$ is a compact metrizable space.
\end{Def}
If $\fa_\theta^1$ denotes the unit sphere, then $\C_i\to \C$ if and only if
the Hausdorff distance between $\C_i\cap \fa_\theta^1$ and $\C\cap \fa_\theta^1$ tends to $0$.

We prove the lower semicontinutity of limit cones under the reductive Zariski closure hypothesis:
 \begin{prop}[Lower semicontinuity]\label{lem:LC}\label{LC}
     Let $\G$ be a finitely generated discrete subgroup of $G$, and let
       $\sigma_n$ be a sequence of discrete representations converging to $\id_\Ga$ in $\Hom(\Ga, G)$.
       Then for any $\theta\subset \Pi$,
        any accumulation point $\L$ of the sequence $\L^\theta_{\sigma_n(\G)}$   in $\mathsf{C}(\fa_\theta^+)$ contains $p_\theta(\L^{\op{Jor}}_\G)$.

   In particular, if
     $\L^\theta_\Ga=p_\theta(\L_\Ga^{\op{Jor}})$ (e.g., the Zariski closure of $\Ga$ is reductive), then $\L$ contains $\L^\theta_\Ga$.
     
 \end{prop}

 \begin{proof} Since we have $\L_{\sigma_{n}(\G)} \supset \L^{\Jor}_{\sigma_{n}(\G)}$  by \eqref{eqn:JLC}, it suffices to prove that
if $\L'$ is an accumulation point of the sequence $p_\theta(\L_{\sigma_n(\G)}^{\op{Jor}})$   in $\mathsf{C}(\fa_\theta^+)$, then $\L'$ contains $p_\theta(\L^{\op{Jor}}_\G)$.
 By passing to a subsequence, we may assume that $p_\theta(\L_{\sigma_n(\G)}^{\Jor} )\to \L'$ as $n\to\infty$. Since $\P(\L^\theta_\G)$ is compact, it is enough to show that for all unit vectors $v\in p_\theta(\L_\G^{\Jor})$, the distance $d_{\fa_\theta}(v,p_\theta(\L_{\sigma_n(\G)}^{\Jor}))$ converges to $0$ as $n\to \infty$.

Let $v\in p_\theta(\L^{\op{Jor}}_\Ga)$ be a unit vector. Then 
there exists a sequence $\g_k\in\G$ such that 
$$\| p_\theta(\lambda'(\g_k))-v \| <\frac{1}k\quad\text{for all $k\ge 1$,}$$
where $\lambda'(\ga_k)$ denotes the unit vector in the direction of $\lambda(\ga_k)$.
     Since the Jordan projection $\lambda:G\to\fa^+$ is continuous,
     for each $k$, we can choose $n_k$ large enough such that
     $\| \lambda'(\g_k) - \lambda'(\sigma_{n_k}(\g_k))\| <\frac{1}k$. Thus
     $$\| p_\theta(\lambda'(\sigma_{n_k}(\g_k)))- v\| < \frac{2}k
    \quad\text{     for all $k\ge 1$.}$$
     
    Since $ \lambda'(\sigma_{n_k}(\g_k)) \in \L^{\Jor}_{\sigma_{n_k}(\G)}$,
    this shows that 
     \begin{equation}\label{eqn:3.1}
         d_{\fa_\theta}(p_\theta(\L^{\Jor}_{\sigma_{n_k}(\G)}), v) \to 0
         \quad \text{as } n\to\infty,
     \end{equation}
proving the claim.
 This finishes the proof together with Theorem \ref{be}.
 \end{proof}
 
Since $\mu(\Ga)$ is not contained in $\L_\Ga$ in general, one cannot replace the Jordan projection $\lambda(\ga_k)$ by the Cartan projection $\mu(\ga_k)$
in the above proof.

\section{Limit cones of Anosov subgroups}
Let $G$ be a connected semisimple real algebraic group. We keep the notation from section \ref{lower}. Fix a non-empty subset $\theta\subset \Pi$ in this section.
\begin{Def}\label{anosov}  For a finitely generated subgroup $\Ga<G$,
a representation $\sigma:\Ga\to G$ is called $\theta$-Anosov
if there exists a constant $C>1$ such that for all $\alpha\in \theta$
and $\ga\in \G$,
\be\label{eqn:ta} \alpha(\mu(\sigma(\ga))\ge C^{-1} |\ga| -C\ee 
where $|\cdot |$ denotes the word metric on $\Ga$ with respect to a fixed finite generating subset of $\Ga$.
We say that $\Ga$ is a $\theta$-Anosov subgroup of $G$ if the natural inclusion map $\op{id}_\Ga$ is $\theta$-Anosov.
\end{Def}
 Each $\theta$-Anosov representation is discrete and has finite kernel \cite{GW}. In particular, all $\theta$-Anosov subgroups are discrete subgroups of $G$.
A $\Pi$-Anosov subgroup is also called a Borel-Anosov subgroup. By an Anosov subgroup, we mean a $\theta$-Anosov subgroup for some non-empty $\theta\subset \Pi$.
As mentioned in the introduction, all Anosov subgroups in our paper are assumed to be non-elementary, i.e., not virtually cyclic.

Since there is a uniform constant $C>0$ such that $\|\mu(g)\| \ge C \alpha (\mu(g)) $ for any $g\in G$ and any $\alpha\in \Pi$,  the inclusion map $(\Ga, |\cdot |)\to  (G, d)$ is a quasi-isometric embedding for any Anosov subgroup $\Ga$,
that is, there exists $L>1$ such that for all $\ga_1, \ga_2\in \Ga$,
\be\label{qi} L^{-1}d(\ga_1, \ga_2) -L \le |\ga_1^{-1}\ga_2| \le L d(\ga_1, \ga_2) +L .\ee

  It is immediate that for a $\theta$-Anosov subgroup
$\Ga < G$, 
 \be\label{kan} \L_\Ga\cap \ker \alpha=\{0\} \quad\text{ for all $\alpha\in \theta$.}\ee

Any real algebraic subgroup $L$ of $G$ admits a Levi decomposition:
$L=H\ltimes U$ where $U$ is the unipotent radical of $L$ and $H$ is a reductive subgroup normalized by $U$ \cite{Bo}.
Without the reductive Zariski closure assumption, we have the following (cf. Theorem \ref{be}). 
\begin{lemma}\label{Jordan}
    For any $\theta$-Anosov subgroup $\G$ of $G$, we have
    $$\L_{\G}^{\theta} = p_\theta (\L_\G^{\rm Jor}).$$

  Moreover if $\overline\Ga^{\op{Zar}}=H\ltimes U$ is a Levi decomposition of the Zariski closure of $\Ga$ and $\Ga_H$
  denotes the projection of $\Ga$ to $H$, then $\Ga_H$ is discrete and $\L_\Ga^\theta=\L_{\Ga_H}^\theta$.   
\end{lemma}

\begin{proof}
Let $G_0= \overline\Ga^{\op{Zar}}=H\ltimes U$.  For $g\in G_0$, write $g_H\in H$ for the $H$-component of $g$, so that $\Ga_H=\{\ga_H:\ga\in \Ga\}$.
   By \cite[Lemma 2.11]{Tsouvalas}, there exists $C>0$ such that for each $\alpha\in\theta$ and each $\g\in\G$,
    \be\label{tt}
    | \omega_\alpha\left( \mu(\g)- \mu(\g_H) \right) |\le C
    \ee 
where $\omega_\alpha$ is the fundamental weight corresponding to $\alpha$ (see, e.g., \cite[2.1]{DKO}). Let  $\ga_i\in \G$ be an infinite sequence going to $\infty$, and $\alpha\in \theta$.
Since $\Ga$ is $\theta$-Anosov,  we have $\alpha(\mu(\ga_i))\to \infty$, and hence $\omega_\alpha(\mu(\ga_i))\to \infty$.  By \eqref{tt}, $\mu((\ga_i)_H)\to \infty$. Therefore
$\Ga_H$ is a Zariski dense {\it discrete} subgroup of $H$. By Theorem \ref{be}, we have
$\L_{\Ga_H}=\L_{\Ga_H}^{\op{Jor}}$.
Since $\lambda(g)=\lambda(g_H)$ for any $g\in G_0$, we have 
$\L_{\Ga_H}^{\op{Jor}}=\L_{\Ga}^{\op{Jor}}$.

Since $\{\omega_\alpha:\alpha\in \theta\}$ forms a basis for $\fa_\theta^*$,
\eqref{tt} now implies that 
$$ \sup_{\ga\in \Ga} \| p_\theta (\mu(\ga) - \mu(\ga_H))\|<\infty. $$
 Since $\L_\Ga^\theta$ and $\L_{\Ga_H}^\theta$ are asymptotic cones of $\Ga$ and $\Ga_H$ respectively, this implies that
  $\L_\Ga^\theta=\L_{\Ga_H}^\theta$. Therefore
  $\L_\Ga^\theta= p_\theta(\L_{\Ga}^{\op{Jor}})$. \end{proof}

The set of $\theta$-Anosov representations is open in $\Hom(\Ga, G)$ \cite{GW}, and hence
if $\Ga$ is $\theta$-Anosov and $\sigma_n \to \id_\Ga$, then $\sigma_n$ is a $\theta$-Anosov (in particular discrete) representation for all sufficiently large $n$.
Hence we get the following lower semicontinutity from Proposition \ref{LC} and Lemma \ref{Jordan}
without any assumption on the Zariski closure of $\Ga$ and the discreteness of $\sigma_n$:
\begin{prop}\label{LCA}
   Let $\G$ be a $\theta$-Anosov subgroup of $G$.
      If $\sigma_n\to \id_\Ga$  in $\Hom(\Ga, G)$, then  any accumulation point $\L$ of the sequence $\L^\theta_{\sigma_n(\G)}$   in $\mathsf{C}(\fa^+)$ contains $\L^\theta_\G$.
\end{prop}

We will need the following corollary of the lower semicontinuity of the limit cones:
\begin{cor} \label{co} Let $\Ga$ be a Zariski dense $\theta$-Anosov subgroup of $G$, and let $\sigma_n \to \id_\Ga$ in $\Hom(\Ga, G)$.
For any closed  cone $\cal C$ contained in $\inte \L_\Ga^\theta\cup\{0\}$,
we have 
$$\cal C\subset\inte \L_{\sigma_n(\Ga)}^\theta\cup\{0\}\quad\text{ for all large $n$.}$$
    \end{cor}

\begin{proof} 
Since the set of Zariski dense representations of $\Ga$ is Zariski open in $\Hom(\Ga, G)$ \cite[Proposition 8.2]{AB} and non-empty since $\Ga$ is Zariski dense,
we may assume without loss of generality that for all $n\ge 1$, $\sigma_n(\Ga)$ is Zariski dense in $G$ , and 
hence $\L_{\sigma_n(\Ga)}^{\theta}$ is a convex cone with non-empty interior by Theorem \ref{be0}.

Let $\cal C$ be a closed cone contained in $\inte \L_\Ga^\theta\cup\{0\}$.
Let $Z=\cal C\cap \fa^1$.
Suppose that for some sequence $ n_k \in \mathbb{N} $ going to $\infty$, we have
    a sequence of vectors $v_k\in  Z-  \operatorname{int} \mathcal{L}_{\sigma_{n_k}(\Gamma)}^{\theta} $. Since $ \mathcal{L}^\theta_{\sigma_{n_k}(\Gamma)} $ is convex, there exists a closed half-space $ H_k \subset \mathfrak{a}_\theta $ whose boundary contains $v_k\in Z $ and  $ \mathcal{L}^\theta_{\sigma_{n_k}(\Gamma)} \subset H_k$.

 Since $Z$ is compact, there exists $\e>0$ such that the $2\e$-neighborhood of $Z$ is contained in $\inte\L_\Ga^\theta$. 
Let $B_k=B(v_k, 2\e)\subset \inte\L_\Ga^\theta$ denote the ball centered at $v_k$ of radius $2\e>0$. Since $v_k$ lies in the boundary of $H_k$, the $\epsilon$-neighborhood of $H_k$ does not contain $B$ and hence  the $\epsilon$-neighborhood of $\mathcal{L}_{\sigma_{n_k}(\Gamma)} \cap \mathfrak{a}^1 $ does not contain $B_k$ and, in particular,  it does not contain $\L_{\Gamma}^\theta$.
Therefore, for any limit $\L$ of the sequence $\mathcal{L}^\theta_{\sigma_{n_k}(\Gamma)}$ in $\mathsf C(\fa_\theta^+)$, the $\epsilon$-neighborhood of $\L \cap \fa^1$ does not contain $\L_\Gamma^\theta \cap \fa^1$. This contradicts Proposition \ref{LCA}. Hence $Z\subset  \operatorname{int} \mathcal{L}_{\sigma_{n}(\Gamma)}^{\theta}  $ and consequently  $\cal C \subset  \operatorname{int} \mathcal{L}_{\sigma_{n}(\Gamma)}^{\theta} \cup\{0\} $ 
 for all large $n$.
\end{proof}

\section{Upper semicontinuity of limit cones} 
Let $G$ be a connected semisimple real algebraic group.
In this section, we prove the upper semicontinuity of the limit cones
of $\theta$-Anosov and $\theta$-convex subgroups of $G$.

The main tool of our proof is the {\em local-to-global principle for Morse quasigeodesics}
due to Kapovich-Leeb-Porti  (\cite[Theorem 1.1]{KLP2}, see Theorem \ref{thm:LTG}). 
In order to explain this principle, we need to
recall some terminology. Recall the notation $\fa_\theta=\bigcap_{\alpha\in\Pi-\theta} \ker\alpha$ and
$\cal W_\theta$ from \eqref{at0}.  
Let $M_\theta$ denote the centralizer of $\exp \fa_\theta$ in $K$.

The notion of a $\theta$-admissible cone is crucial:
\begin{Def} Let $\theta=\i(\theta)$. A closed cone $\C\subset \fa^+$ is called {\em $\theta$-admissible} if
\begin{enumerate}
    \item $\C$ is $\i$-invariant: $\i(\C)=\C$;
    \item $\C$ is $\theta$-convex, i.e., $\cal W_\theta \C$ is convex;
    \item $\C\cap \left(\bigcup_{\alpha\in \theta} \ker \alpha\right) =\{0\}$.
\end{enumerate}
\end{Def}

\subsection*{Local-to-global principle} Let $\C$ be a
 $\theta$-admissible cone of $\fa^+$. We denote by $X$ the  Riemannian symmetric space $G/K$ equipped with the  metric $d$ induced from $\langle\cdot,\cdot\rangle$. 
The notation $d(\cdot,\cdot)$ will denote both the left $G$-invariant Riemannian distance function on $X$. Let $o=[K]\in X$.
 Let $x $ and $y$ be a $\C$-regular pair in $X$, that is,
$$\mu(g_1^{-1}g_2) \in \C$$ for any $g_1,g_2\in G$ such that $x = g_1 o$ and $y = g_2 o$.
The {\em $\C$-cone} emanating from $x$ and passing through $y$ is defined as
 $$V_\C(x,y) \coloneqq g M_\theta (\exp \C)o$$ where $g \in G$ is any element such that $x = go$ and $y \in g (\exp \C) o$.  Noting that $V_\C (x, y)\cap g A o= g\exp (\W_\theta. \C)o$,
 the convexity of the cone $V_\C(x,y)$ implies the $\theta$-convexity of $\C$.
The main point of a $\theta$-admissible condition on the cone $\C$ is that for a
$\theta$-admissible cone $\C$,  all $\C$-cones are convex subsets of $X$ \cite{KLP}.

The {\em $\C$-diamond} connecting $x$ to $y$ is defined as 
\be\label{dia}
 \Diamond_\C(x,y) = V_\C(x,y) \cap V_\C(y,x).
\ee

 Cones and diamonds can be viewed as higher rank generalizations of geodesic rays and segments in the hyperbolic space.

\begin{Def}
   Let  $I\subset \R$ be an interval, and
 $L\ge 1$ and $D\ge  0$ be constants. 
\begin{enumerate}
    \item 
A map $f: I \to X$ is called {\em $(\C,D,L)$-Morse}, if $f$ is an $L$-quasi-isometric embedding\footnote{That is, $L^{-1}|t-s|-L \le d(f(t), f(s))\le L |t-s| +L $ for all $t,s\in I$.}
and for all $s\le t$ in $I$, the image 
$f([s,t])$ lies in the $D$-neighborhood of 
 some $\C$-diamond $\Diamond_{\C} (x,y)$ in $X$, where the tips $x$ and $y$ are a $\C$-regular pair 
 such that $d(x, f(s))\le D$ and $d(y, f(t))\le D$.

\item 
For $S\ge 1$, a map $f: I\to X$ is called {\em $(\C,D,L,S)$-local Morse} if for all $s,t\in I$ with $0\le t-s \le S$, the restriction map $f\vert_{[s,t]}$ is $(\C,D,L)$-Morse.

\item More generally, for a geodesic Gromov hyperbolic space $Y$, a map $f: Y \to X$ is called $(\C,D,L)$-Morse (resp. $(\C,D,L,S)$-local Morse) if the restriction of $f$ to each unit-speed parametrized geodesic in $Y$ is $(\C,D,L)$-Morse (resp. $(\C,D,L,S)$-local Morse).

\end{enumerate}
\end{Def}

We remark that the Morse datum is invariant under post composition by any element $g\in G$, considered as an isometry of $X$.
More precisely, if $f: Y \to X$ is $(\C,D,L)$-Morse (resp. $(\C,D,L,S)$-local Morse), then so is $g\circ f$ for any $g\in G$.

Finally, we can state the local-to-global principle for Morse quasigeodesics:

\begin{theorem} [{Kapovich-Leeb-Porti, \cite[Theorem 1.1]{KLP2}}] \label{thm:LTG}
Let $\theta=\i(\theta)$. Let $\C,\C' \subset\fa^+$ be $\theta$-admissible closed cones with nonempty interior such that $\C - \{0\}\subset {\rm int}\, \C'$. For any interval $I\subset \R$ and any constants  $L\ge 1$ and $D\ge 0$, there exist $L',S\ge 1$ and $D'\ge 0$ such that every $(\C,D,L,S)$-local Morse map $$f: I \to X$$ is $(\C',D',L')$-Morse.
\end{theorem}

\subsection*{Upper semicontinuity of limit cone} In order to apply \Cref{thm:LTG} in our setting, we will need the following lemma on the existence of $\theta$-admissible cones:
\begin{lemma}\label{con1} Let $\theta=\i(\theta)$. Let $\mathcal D$ be  a $\theta$-admissible cone of $\fa^+$. Then for any open cone $\C_0$ in $\fa^+$ containing $\mathcal D-\{0\}$,   there exists a $\theta$-admissible closed cone $\C$ such that 
$$\cal D-\{0\} \subset \op{int} \C \subset \C \subset \C_0 \cup\{0\}.$$
\end{lemma}

\begin{proof} 
We observe that  $\cal W_\theta {} \fa^+ $ is equal to the union of all Weyl chambers containing
  $\fa_\theta^+$, and hence is a convex cone.
  
Let $\alpha\in \theta$.  Since  $\ker \alpha \cap \fa^+$ is contained in the boundary of $\cal W_\theta {} \fa^+$, it follows from the convexity of $\cal W_\theta{} \fa^+$ that
$\cal W_\theta {} \fa^+$ is contained in the half space $\{\alpha\ge 0\}.$ Hence both $\inte \fa^+$ and  $\cal W_\theta {} \cal D -\{0\}$ are contained in the open half-space $\{\alpha>0\}$.  Since $\cal D$ is a closed convex cone disjoint from $\ker \alpha-\{0\}$, we may assume without loss of generality that the closure of $\C_0$
is   disjoint from $\ker \alpha-\{0\}$.
Therefore
we can find a linear form $h_\alpha\in \fa^*$ such that
$$(\ker \alpha\cap \fa^+ ) -\{0\} \subset \{h_\alpha<0\} \quad\text{and}\quad 
(\cal C_0\cup  \cal W_\theta {} \cal D) -\{0\}\subset 
\{h_\alpha >0\} .$$

Now set 
\be\label{hhh} H\coloneqq\bigcap_{\alpha\in \theta, w\in \cal W_\theta} \{h_\alpha\circ \op{Ad}_{w} \ge 0\},\ee  which is clearly a $\cal W_\theta$-invariant convex cone. By our choice of $h_\alpha$, the interior
$\inte H$ contains $\cal D-\{0\}$. Since $\theta=\i(\theta)$, we have that $\i(H)$ is also a $\cal W_\theta$-invariant convex cone whose interior contains $\cal D-\{0\}=\i(\cal D)-\{0\}$.

Since $\cal D-\{0\}\subset \C_0$, we can find $\e>0$ such that
the cone $\cal D_\e:=\br_+(\cal N_\e(\fa^1\cap \cal D))$ is contained in $ \C_0\cup\{0\}$ where
 $\fa^1=\{\|w\|=1\}$ is the unit sphere and 
$\cal N_\e  (\fa^1\cap \cal D)=\{w\in \fa^1: \|w- (\fa^1\cap \cal D) \|\le \e\}$.

Define \[\C\coloneqq 
\cal D_\e \cap \fa^+\cap  H \cap \i(H).\]
By construction, we have $
 \C \cap \left(\bigcup_{\alpha \in \theta} \ker \alpha \right)= \{0\}$.
Moreover, $\C$ is a  closed cone in $\fa^+$ whose interior contains $\cal D-\{0\}$.
 Since $\i(\cal D)=\cal D$ and $\i$ preserves the norm on $\fa$, we
have $\i(\cal C)=\cal C$. 
Since $\W_\theta$ acts by isometries on $\fa^1$, we have 
$\W_\theta(\cal D_\e)
=\br_+ \cal N_\e (\cal W_\theta \fa^1\cap \cal W_\theta \cal D)$. By the hypothesis that $\W_\theta\cal D$ is convex,  $\W_\theta(\cal D_\e)$ is convex as well.
Since  $$\cal W_\theta {} \C = \cal W_\theta {}  (\cal D_\e ) \cap \cal W_\theta (\fa^+) \cap H \cap \i(H),$$ it follows that $\cal W_\theta{}  \C $ is convex. 
    Therefore $\C$ is $\theta$-admissible.   
\end{proof}

The next proposition is the main ingredient of the upper semicontinuity of limit cones, which does not yet require the $\theta$-convexity of $\Ga$:
 \begin{prop}\label{ucp} Let $\theta=\i(\theta)$.
     Let $\G$ be a  $\theta$-Anosov  subgroup of $G$. Let $\cal D\subset \fa^+$ be a $\theta$-admissible cone containing $\L_\Ga$. Then
     for any open cone $\C_0\subset \fa^+$ containing
     $\cal D-\{0\}$,
      there exists an open neighborhood $\cal O\subset \Hom(\Ga, G) $ of $\id_\Ga$ such that for all $\sigma\in \cal O$, we have
      $$\L_{\sigma(\Ga)}-\{0\} \subset \C_0.$$
 \end{prop}

\begin{proof} 
   By Lemma \ref{con1}, we can choose $\theta$-admissible closed cones 
    $\C$ and $\C'$ contained in $ \C_0\cup\{0\}$
    so that 
    $$\cal D-\{0\}\subset \inte \C\subset \C \subset \inte \C'\cup\{0\} .$$
    
    Consider a left-invariant word metric $|\cdot |$ on $\G$ with respect to a fixed
    symmetric finite generating set for $\G$. Since $\G$ is $\theta$-Anosov, there exist constants $L, D\ge 1$ such that the orbit map 
    $$\text{$\G \to X$, $\g\mapsto \g o$ }$$ is a $(\C, D,L)$-Morse embedding.
    Let $L',S\ge 1$ $D'\ge 1$ be as in Theorem \ref{thm:LTG} corresponding to the data $\C,\C',D+1,L+1$.

Denote by $\mathfrak o_n:\Ga\to X$ the orbit map of $\Ga$ via $\sigma_n$:
    \[
     \g \mapsto \sigma_n(\g) o.
    \]

\medskip
\noindent{\bf Claim A.} For all sufficiently large $n$, the orbit map $\mathfrak o_n$ is $(\C',D',L')$-Morse.
    
 \begin{proof}[Proof of claim]   Consider the $S$-ball $B_\Ga (e, S)\subset \G$ with respect to the word metric centered at the identity element $e$. Let $f: I\cap\Z \to B_\Ga (e, S)$ be a geodesic, where $I\subset \R$ is an interval. Since  $\sigma_n$ converge to  $\id_\Ga $ in $\Hom(\Ga, G)$, there exists $n_0\ge 1$ such that for all $n\ge n_0$,
    \[
     \max_{\g\in B_\Ga (e, S)}d(\g o, \sigma_n(\g) o) \le \frac{1}{2}.
    \]

Since the orbit map $\ga\mapsto \ga o$ is $(\C, D,L)$-Morse, this implies that the restriction $\mathfrak o_n$ to $B_\Ga (e, S)$ 
    is $(\C, D+1,L+1)$-Morse for all $n\ge n_0$.

    It follows that for all $n\ge n_0$, the map $\mathfrak o_n $ is $(\C, D+1,L+1,S)$-local Morse.
     Indeed, let $s\le t$ be such that $t-s\le S$, and let
     $f: [s,t]\cap\Z \to \G$ be a geodesic. Without loss of generality, we may assume that $t, s\in \Z$.
     Consider the post composition of $f$ with the left multiplication by $f(s)^{-1}$,  which is an isometry on $(\Ga, |\cdot |)$.
    Note that the image of $f(s)^{-1} \circ f$ lies in $B_\Ga (e, S)$. Therefore the map
    \[
     \mathfrak o_n \circ f(s)^{-1} \circ f : [0,S']\cap\Z \to X
    \]
    is $(\C, D+1,L+1)$-Morse. Thus
    $$\mathfrak o_n\circ f = \sigma_n(f(s)) \circ \mathfrak o_n \circ f(s)^{-1} \circ f$$ is $(\C, D+1,L+1)$-Morse. This implies that $\mathfrak o_n $ is $(\C, D+1,L+1,S)$-local Morse. Therefore, by applying Theorem \ref{thm:LTG}, we obtain that the orbit map $\mathfrak o_n$ is $(\C',D',L')$-Morse for all $n\ge n_0$, proving the claim. \end{proof}

Let $B$ denote the ball $B\coloneqq\{g\in G: d(go, o)\le D'\}$. By \cite[Proposition 5.1]{Be1},
there exists a uniform constant $Q>0$ such that
for all $g\in G$ and $b_1, b_2\in B$, 
\be\label{dd} \|\mu(g)-\mu(b_1gb_2)\| \le Q .\ee

 Let $\g\in \Ga$ and $n\ge n_0$. We claim that there exists $u\in \C$, depending on $\g$ and $n$, such that
 \be\label{cl}  \| \mu(\sigma_n(\ga))- u \| \le Q .\ee

By Claim A,
 the point $\sigma_n(\g )o$ lies within the $D'$-neighborhood of some $\C'$-diamond $\Diamond_{\C'}(x_n, y_n)$, where $x_n,y_n$ are a $\cal C'$-regular pair such that
 $d(x_n, o)\le D'$ and $d(y_n, \sigma_n(\ga) o)\le D'$. Let $h_n\in G$ be such that $h_n o$ is an element of the diamond $\Diamond_{\C'}(x_n, y_n)$ such that $d(\sigma_n (\g ) o, h_n o) \le D'$, that is,
  $$ \sigma_n (\ga)=h_n b_n\quad\text{ for some $b_n\in B$.}$$ By the description of the diamond,  
 there exists $g_n\in G$ such that $x_n= g_n o$ and $h_n o = g_n (\exp u) o$ for some $u\in {\cal C}'$. Hence $$\sigma_n (\ga)= h_n b_n = g_n (\exp u ) k_n b_n $$ for some $k_n\in K$.
 
Since $g_n $ and $ k_n b_n$ belong to  $B$, we have by \eqref{dd} that
$$\| \mu(\sigma_n(\ga)) - u\| \le Q,$$
        proving the claim \eqref{cl}. Therefore  for all $n\ge n_0$, $\mu(\sigma_n(\Ga))$ is contained in the $Q$-neighborhood of $\C'$ and hence
$$\L_{\sigma_n(\G)} \subset \C'.$$ Since $\C'-\{0\}$ lies in $\C_0$,
we have $\L_{\sigma_n(\G)} -\{0\}\subset \C_0$ for all $n\ge n_0$. This finishes the proof of the proposition.
\end{proof}

\begin{Def} \label{convex}  
A discrete subgroup $\Ga$ of $G$ is called $\theta$-convex
if the orbit $$\W_{\theta\cup\i(\theta)}\L_\Ga =\bigcup_{w\in \cal W_{\theta\cup\i(\theta)}} \op{Ad}_w \L_\Ga \text{ is a convex subset of }\fa.$$
\end{Def}
If $\theta=\Pi$, then $\W_{\Pi}=\{e\}$, and hence the $\Pi$-convexity of $\Ga$ is same as the convexity of the limit cone $\L_\Ga$.
Since $\W_{\theta\cup\i(\theta)}\L_\Ga \cap \fa^+=\L_\Ga$,
the $\theta$-convexity of $\Ga$ is a stronger condition than the convexity of the limit cone $\L_\Ga$ in general.

\begin{theorem}[Upper semicontinuity]\label{lem:UC}\label{UC} 
     Let $\G$ be a  $\theta$-Anosov $\theta$-convex subgroup of $G$ for some $\theta\subset \Pi$.
     If $\sigma_n\to \id_\Ga$ in $\Hom(\Ga, G)$,   
      then  any accumulation point $\L$ of the sequence $\L_{\sigma_n(\G)}$ in $\mathsf{C}(\fa^+)$ is contained in  $ \L_\G$.
 \end{theorem}
\begin{proof} If $\Ga$ is $\theta$-Anosov and $\theta$-convex,  then  $\Ga$ is $\theta\cup\i(\theta)$-Anosov and $\theta\cup\i(\theta)$-convex and
$\L_\Ga$ is $\theta\cup\i(\theta)$-admissible. Hence
the claim follows from Proposition \ref{ucp} by setting $\cal D=\L_\Ga$.
\end{proof}
Theorem \ref{m1} in the introduction is a consequence of Proposition \ref{LC} and Theorem \ref{UC}.

For Borel Anosov subgroups, we  need neither the reductive Zariski closure assumption nor the convexity assumption in Theorem \ref{m1}.
We first observe:
\begin{lemma}\label{convex2}
Any $\Pi$-Anosov subgroup of $\Ga$  is $\Pi$-convex.
\end{lemma}
\begin{proof} Let $\Ga$ be a Borel Anosov subgroup of $G$.
By Lemma \ref{Jordan} for $\theta=\Pi$, we have $\L_\Ga^{\Jor}=\L_\Ga$. Let $G_0$ be the identity component of the Zariski closure of $\Ga$.   Using the same notation as in the proof of Lemma \ref{Jordan}, we have $\L_\Ga=\L_{\Ga_H}$. Since $\Ga_H$ is Zariski dense in $H$ and $\L_{\Ga_H}=\L_{\Ga}\subset \inte\fa^+\cup\{0\}$ as $\Ga$ is Borel Anosov, $\L_{\Ga_H}$ is convex by Lemma \ref{convex1}. Hence $\L_\Ga$ is convex. \end{proof}

\begin{prop}\label{borel}
    If $\Ga$ is a Borel Anosov subgroup of $G$, then the map 
    $$\sigma \mapsto
\L_{\sigma(\Ga)} \in \mathsf C(\fa^+)$$ is continuous at  $\id_\Ga \in \op{Hom}(\Ga, G)$.
\end{prop}
\begin{proof}
    The lower semicontinuity follows from Proposition \ref{LCA} with $\theta=\Pi$. The upper semicontinuity follows from Theorem \ref{UC} and Lemma \ref{convex2}.
\end{proof}

\section{Application to sharpness and examples of $\theta$-convex subgroups}\label{ex}
Let $G$ be a connected semisimple real algebraic group.
Recall from the introduction that
for a closed subgroup $H$ of $G$, a discrete subgroup $\Ga$ of $G$ is called {\it sharp} for $G/H$ if $\L_\Ga\cap \L_H=\{0\}$. We show that the sharpness is an open condition in the following situation:
\begin{prop} \label{sharp}
 Let $H<G$ be a closed subgroup. Let $\theta\subset \Pi$.
 Let $\Ga$ be a
 $\theta$-Anosov subgroup of $G$ which is sharp for $G/H$.
 Suppose  that $\Ga$ is $\theta$-convex, or more generally that there exists a $\theta\cup \i(\theta)$-admissible closed cone $\cal C\subset \fa^+$ 
 such that $$\text{ $\L_\Ga\subset \C\quad $ and  $ \quad \cal C\cap \cal \L_H=\{0\}$.}$$
Then there exists an open neighborhood $\cal O$ of $\id_\G$ in $\Hom(\Ga, G)$
such that for all $\sigma\in \cal O$, $\sigma(\Ga)$ is sharp for
$G/H$.
\end{prop}
\begin{proof} By the hypothesis, we may assume that $\theta=\i(\theta)$.
Since $\C $ is a $\theta$-admissible cone
such that $\C- \{0\}$ is contained in the open cone $\fa^+ - \L_H$ in $\fa^+$, by \Cref{con1}, there exists a $\theta$-admissible cone $\C_1$ in $\fa^+$ such that 
$$\C -\{0\} \subset \inte \C_1\subset \fa^+-\L_H.$$
Since $\L_\Gamma \subset \inte \C_1\cup\{0\}$, by \Cref{ucp}, we have that for all $\sigma \in \Hom(\G,G)$ sufficiently close to $\id_\Ga$, $\L_{\sigma(\Gamma)}-\{0\} \subset \inte\C_1$, implying that $\L_{\sigma(\Ga)}\cap \L_H=\{0\}$. If $\Ga$ is $\theta$-convex, $\L_\Ga$ is a $\theta$-admissible closed cone. Hence this proves the claim. \end{proof}

The following is then a special case of Proposition \ref{sharp}: note that the $\theta$-convexity of $\Ga$ is not assumed.
\begin{cor}\label{sharp2}
 If $H$ is a reductive algebraic subgroup  with co-rank\footnote{The co-rank of $H$ is defined as $\op{rank}G-\op{rank}H$.} one in $G$ and $\Ga$ is $\theta$-Anosov and sharp for $G/H$ for some $\theta\subset \Pi$, then
 $\sigma(\Ga)$ is sharp for
$G/H$ for all $\sigma\in \Hom(\Ga, G)$ sufficiently close to $\id_\Ga$. 
\end{cor}
\begin{proof}  We may assume without loss of generality that $ \theta $ is a maximal subset for which $ \Ga $ is $\theta$-Anosov. In particular, $\theta=\i(\theta)$. Since $\L_{H, H}=\fa_H^+$, it follows from Lemma \ref{hg} that we have
$$\L_H=\W \fa_H\cap \fa^+ .$$
 We claim that
\[
 \W_\theta \L_\G  \cap \fa^1\quad \text{is connected}
\] where $\fa^1$ is the unit sphere in $\fa$.
To see this, note that $ \L_\G \cap \fa^1 $ is connected \cite[Proposition A.2]{DR}. By the maximality assumption on $\theta$, we have that $\L_\Ga\cap   \ker \alpha \ne \{0\} $ for every $ \alpha \in \Pi - \theta $.
 In particular, if we denote by $ w_\alpha $ the reflection along the wall $\ker \alpha$, then for all $\alpha\in \Pi-\theta$,
\[
(\L_\G \cap \fa^1) \cup w_\alpha (\L_\G \cap \fa^1)
\]
is connected, since $ w_\alpha $ fixes the intersection $ \L_\G \cap \ker\alpha\cap  \fa^1 $ pointwise.
Since $ \W_\theta $ is generated by $ \{ w_\alpha : \alpha \in \Pi - \theta\} $,
 the connectedness of $ \W_\theta \L_\G \cap \fa^1 $ follows.

Since $\Ga$ is sharp for $G/H$, we have $ \L_\Ga \cap \L_H = \{0\} $.
Since the co-rank of $H$ is one, $\fa_H$ is a hyperplane in $\fa$. Moreover, since $ \L_\G - \{0\} $ is connected,  and $\L_H=\W \fa_H\cap \fa^+$, the set $\L_\Ga - \{0\} $ must lie in one of the connected components of $ \W_\theta \fa^+ - \W \fa_H $, say $ \C $. Note that $ \C $ is a convex open cone in $ \W_\theta \fa^+ $ containing $\L_\Ga-\{0\}$.
Since $\W_\theta\L_\Ga - \{0\}$ is connected, it follows that
$$\W_\theta \L_\Ga -\{0\} \subset \C  .$$  For each $w\in \W_\theta$,
since $w\C$ is a connected component of $\W_\theta\fa^+ -\W \fa_H$ containing $w(\W_\theta \L_\Ga - \{0\}) = \W_\theta \L_\G - \{0\}$, we must have
$w\C=\C$, and hence $$\W_\theta \C=\C.$$ In particular,  the closure $ \bar \C $ of $ \C $ in $ \fa $ is a $ \W_\theta $-invariant closed convex cone. 

Let $$ \C_0 := \fa^+ \cap \bar \C ;$$ thus, $ \W_\theta \C_0 = \C $. Note that $ \inte \C_0 $ is a connected component of $ \fa^+ - \W \L_H $ containing $ \L_\Ga - \{0\} $.
Since $ \L_H $ and $ \L_\Ga $ are both $ \i $-invariant, it follows that $ \C_0 $ is also $ \i $-invariant.  Let $\cal D$ be any $\theta$-admissible cone containing $\L_\Ga$ (see Lemma \ref{theta} for the existence of such a cone). Then the cone
$$ \C_1 :=\C_0\cap \cal D$$
is a $ \theta $-admissible cone containing $ \L_\G - \{0\} $ in its interior, so the second claim follows from the first one.   
\end{proof}
As mentioned in the introduction, Corollary \ref{sharp2} was obtained in \cite[Corollary 6.3]{KT} by a different approach.

\subsection*{Examples of $\theta$-Anosov and $\theta$-convex subgroups}
Since $\W_{\theta\cup\i(\theta)}$ fixes
$\fa_{\theta\cup\i(\theta)}^+$, we have the following $\theta$-version of Theorem \ref{convex2}:
\begin{lemma} Any $\theta$-Anosov subgroup $\G<G$ such that  $\L_\Ga$ is a convex cone contained in $\fa_{\theta\cup\i(\theta)}^+$ is $\theta$-convex. 
\end{lemma}
    As another example, we have the following:
    \begin{lemma}\label{one}
    If $H$ is a rank one simple algebraic subgroup of $G$ and $\Ga$ is a non-elementary convex cocompact subgroup of $H$, then $\Ga$ is a $\theta$-Anosov $\theta$-convex subgroup of $G$ for some $\theta\subset \Pi$.
\end{lemma}
\begin{proof} We note that Guichard-Wienhard \cite[Proposition 4.7]{GW} showed that $\G$ is a $\theta$-Anosov subgroup of $G$ for some $\theta$. However, since we also need to demonstrate that $\Ga$ is $\theta$-convex for an appropriate choice of $\theta$, we provide a complete proof of this result.

 If the rank of $G$ is one, then $\G$ is a convex cocompact subgroup of $G$. So the claim just follows since any non-elementary convex cocompact subgroup $\G$ of a rank one Lie group $G$ is $\Pi$-Anosov and
   $\L_\Ga=\fa^+$. 
   
   Hence we assume that $\op{rank} G\ge 2$.
   Since $H< G$ is semisimple,  we may assume that
   we have a Cartan decomposition $H=(K\cap H)\fa_H^+ (K\cap H)$ as in \eqref{ch}.
  If $v$ is a unit vector in $\fa_H^+$, since $\exp \br_+v$
   and $\exp \br_-v$ are conjugate by an element of $K\cap H$,
we have that $\mu(H)=\L_\Ga$ is a  ray.

Set 
$$\theta\coloneqq\{\alpha\in \Pi: \L_\Ga \cap \ker \alpha = 0\} .$$
Since $\# \Pi \ge 2$ and $\L_\Ga$ is a ray, $\theta\ne \emptyset$.
Since $\L_\Ga$ is $\i$-invariant,
 $\theta=\i(\theta)$.  Moreover, by definition 
 $$\L_\Ga \subset \bigcap_{\alpha\in \Pi-\theta} \ker\alpha = \fa_\theta.$$
 Since $\W_\theta$ fixes $\fa_\theta$ pointwise,  we have $\W_\theta \L_\Ga=\L_\Ga$ which is hence  a ray.
 
Since  $\G< H$ is a convex cocompact subgroup, then
$(\G, |\cdot|)$ is quasi-isometrically embedded in $H/(K\cap H)$ under the orbit map $\ga \mapsto \ga (K\cap H) $. Since $H/(K\cap H)\subset X=G/K$ is an isometric embedding, the orbit map
$\ga \mapsto \ga o$ is a quasi-isometric embedding of $(\Ga, |\cdot|)$ into $(X, d)$. Since $\mu(\Ga)$ lies in a ray $\mu(H)=\L_\Ga$ which is disjoint from $\bigcup_{\alpha\in \theta} \ker \alpha$ except at zero,
it satisfies  \eqref{eqn:ta} for all $\alpha\in\theta$. Therefore $\Ga$ is $\theta$-Anosov and $\theta$-convex.
 \end{proof}

As a third family of examples, we have the following:
\begin{lemma}\label{two}
 Let $\op{rank} G=2$ and $\Ga<G$ be an Anosov subgroup. 
If $ \theta $ is the maximal subset of $\Pi$ for which $ \G $ is $ \theta$-Anosov, then
  $ \G $ is $\theta $-convex. 
\end{lemma}
\begin{proof} Since $\L_\G \cap \fa^1$ is connected \cite[Proposition A.2]{DR} and $\text{dim} (\fa )=2$,
$\L_\Ga$ is a convex cone. If $ \G $ is Borel-Anosov, the claim follows from \Cref{convex2}.

Otherwise, $\G$ is $ \theta $-Anosov with $ \theta $ as a singleton, but not Borel-Anosov. By the maximality, $\theta=\i(\theta)$.
Let $\alpha$ denote the simple root in $\Pi-\theta$. 
We claim that $ \L_\G $ contains the ray $ \fa_\theta^+ = \fa^+\cap \ker \alpha$. Suppose not. Then there exists a closed cone $\C \subset \fa^+$ such that $\L_\G \subset \inte \C \cup \{0\}$ and $\C \cap \fa_\theta^+ = \{0\}$. In this case, $\mu(\g) \in \C$ for all but finitely many $\g \in \G$ (cf. \cite[Lemma 4.4]{DKO}), and hence 
 $$\liminf_{\g \in \G-\{e\}} \frac{\alpha(\mu(\g))}{\|\mu(\g)\|} > 0.$$
Since $\G$ is quasi-isometrically embedded in $G$,
this implies 
\[
 \liminf_{\g \in \G-\{e\}} \frac{\alpha(\mu(\g))}{|\g|} > 0.
\]  
 In particular, it follows that $\G$ is $\{\alpha\}$-Anosov. Since $\G$ is $\theta$-Anosov and $\Pi =\{\alpha\}\cup\theta$, we see that $\G$ is Borel-Anosov, which contradicts our assumption.

Therefore, $\L_\Ga$ is a convex cone containing the ray $\fa_\theta^+$ in the boundary. Therefore  $\W_\theta \L_\Ga$ is convex.
\end{proof}

Therefore the discussion above together with Proposition \ref{borel} proves Corollary \ref{three0} in the introduction. 

\section{Failure of the continuity of limit cones: an example}\label{fa}
Upper semicontinuity of limit cones does not hold in general-even for Anosov subgroups -- unless the subgroup
satisfies the convexity property.
In this section, we give an example that illustrates this failure. These examples are 
 $\theta$-Anosov subgroups of $\SL(4, \br)$ but are not $\theta$-convex. 

In  this section, we set
$$G= \SL(4,\R).$$
We fix the positive Weyl chamber $\fa^+$ of $G$ as follows:
$$\fa^+=\{ v=\text{diag}(v_1, v_2, v_3, v_4): v_1\ge v_2\ge v_3\ge v_4, \; \sum_{i=1}^4 v_i=0\} .$$
For simplicity, we identify $v=\text{diag}(v_1, v_2, v_3, v_4)$ with $v=(v_1, v_2, v_3, v_4)$.
The set $\Pi$ of simple roots  is given by
$$\alpha_i(v)=v_i-v_{i+1}\quad\text{for  $1\le i\le 3$} .$$
Consider the following block diagonal subgroup:
\begin{equation*}
 H= \begin{pmatrix}
     \SL(3, \R)  & 0\\
     0 & 1
 \end{pmatrix} < \SL(4, \R).   
\end{equation*}

Observe that the Cartan projection $\mu(H)=\mu_G(H) \subset \fa^+$ is a folded plane (see \Cref{folded_plane}):
$$\mu(H)=V_1\cup V_2$$
where
$$ V_1=\{ (v_1, v_2, 0, v_3)\in \fa^+: v_1\ge v_2\ge 0\} $$
and
$$V_2=\{(v_1, 0, v_2,v_3)\in \fa^+: v_1\ge 0\ge v_2\} . $$

\begin{figure}
\begin{tikzpicture}[scale=2]
    \path[fill=gray!10, thick] (0, 0) -- (3, 0) -- (1.5, 1.5) -- cycle;
    \draw[line width=0.4mm, LimeGreen!70] (2.25, 0.75) -- (1.5, 0.01) -- (0.75, 0.75);
    \draw[line width=0.4mm, PineGreen] (2, 0.5) -- (1.5, 0.01) -- (1, 0.5);
    \draw[thick] (0, 0) -- (3, 0) -- (1.5, 1.5) -- cycle;
    \node[below] at (1.5, -0.) {$\op{ker}\alpha_2$};
    \node[left] at (0.73, 0.78) {$\op{ker}\alpha_1$};
    \node[right] at (2.23, 0.78) {$\op{ker}\alpha_3$};
\end{tikzpicture}
\caption{The folded ``plane'' in the unit sphere of $\fa^+$} (green), containing the limit cone (darker green)\label{folded_plane}
\end{figure}
More precisely, 
consider the Cartan projection of $H$:
$\mu_H: H\to \fa_H^+$ where
$$\fa_H^+=\{(v_1, v_2, v_3, 0): v_1\ge v_2\ge v_3, v_1+v_2+v_3=0\}. $$

If $\mu_H(h)=(v_1, v_2, v_3,0)$ for $h\in H$,
then $\mu(h)\in V_1\cup V_2$ is given by \begin{equation*}
   \mu(h)= \begin{cases} (v_1, v_2, 0, v_3) &\text{ if $v_2\ge 0$}
   \\ (v_1, 0, v_2,v_3) &\text{ if $v_2\le 0$}. \end{cases}
\end{equation*}
Also consider the ray
$$V_0=V_1\cap V_2=\{(v_1, 0, 0, -v_1): v_1\ge 0\}.$$

\begin{lemma} \label{ta} Let $\theta=\{\alpha_1, \alpha_3\}$.
Any  Zariski dense Borel-Anosov subgroup $\Ga$ of $H$ is
a $\theta$-Anosov subgroup of $G$ which is not $\theta$-convex.  
Moreover,  for each $i=1,2$,
\be\label{lg} \L_\Ga\cap V_i\not\subset V_0 .\ee 
\end{lemma}
\begin{proof} Fix a word metric $|\cdot |$ on $\G$ with respect to a finite generating subset of $\Ga$.
Since $\Ga$ is a Borel-Anosov subgroup of $H$,
there exists  $C>1$ such that
for any $\ga\in \Ga$, 
$$\alpha_1(\mu_H(\ga))\ge C^{-1}|\ga|-C .$$
For $\ga\in \Ga$ with $\mu_H(\ga)=(v_1, v_2, v_3,0)$,
we then have 
$\alpha_1(\mu(\ga))=v_1-v_2$ if $\mu(\ga)\in V_1$ or 
 $\alpha_1(\mu(\ga))=v_1$ if $\mu(\ga)\in V_2$.
In the second case,
we have $0\ge v_2\ge v_3=-v_1-v_2$, and hence  $v_1\ge -2v_2\ge -v_2$.
So $v_1\ge \frac{1}{2}(v_1-v_2)$.
Therefore for any $\ga\in \Ga$,
$$\alpha_1(\mu(\ga)) \ge\frac{1}{2} \alpha_1(\mu_H(\ga)) \ge \frac{1}{2C} |\ga|- 2C .$$

It follows that $\G$ is an $\{\alpha_1\}$-Anosov subgroup of $G$ or, equivalently, an $\{\alpha_1, \alpha_3\}$-Anosov subgroup of $G$
since $\i(\alpha_1)=\alpha_3$.

We now claim that $\Ga$ is not $\theta$-convex for $\theta=\{\alpha_1, \alpha_3\}$.
Since $\Ga$ is a Zariski dense Borel-Anosov subgroup {\it of} $H$, $\L_{\Ga, H}$ is a convex cone with non-empty interior and contained in $\inte\fa_H^+\cup\{0\}$. Hence
$\L_\Ga-\{0\}$ is contained in the union of relative interiors of $V_1$ and $V_2$, $\L_\Ga\cap V_i \not\subset V_0$ and $\L_\Ga$ intersects each relative interior of $V_i$ non-trivially.  Hence $\L_\Ga$ is not convex, and as a consequence, $\G$ is not $\theta$-convex. \end{proof}

Let $\G$ be a Zariski-dense Schottky subgroup of $H$
generated by two diagonalizable elements $$a= \exp v= \text{diag}(e^{v_1}, e^{v_2}, e^{v_3}, 1)$$
and $$ b=g \operatorname{diag} (e^{w_1}, 1, e^{-w_1}, 1)g^{-1} \quad\text{for some $g \in H$}$$
where $v_1>v_2>0> v_3$ and $w_1>0$.
We refer to \cite{ELO} for a precise definition of a higher rank Schottky subgroup and for the proof that a Schottky subgroup is a Borel-Anosov subgroup.

\begin{prop} \label{ce} For the above Zariski dense Schottky subgroup $\Ga<H$, the map 
$$\sigma\mapsto \L_{\sigma(\Ga)}$$
is not continuous at $\id_\Ga\in \Hom(\Ga, G)$.
\end{prop}
\begin{proof} Since the diagonal entries of $a=\exp v $ are all distinct, $a$ is a loxodromic  element of $G=\SL(4,\br)$.
It follows from
  \cite[Proposition 4.4]{Tits_free} that  the union of all Zariski-closed and Zariski-connected proper subgroups of $G$ containing  $a$ is contained in a proper Zariski-closed subset of $G$, say $\cal Z$.
Since $\Ga$ is not Zariski dense in $G$, we have 
$$b\in \cal Z.$$ Since $G-\cal Z$ is Zariski open and the set of all loxodromic elements of $G$ is Zariski dense (\cite{Pr}, \cite[Corollary 1]{Be}),
we can find a sequence
 $b_n\in G-\cal Z$ of loxodromic elements converging to $b$.

Since $\Ga$ is a free group generated by $a$ and $b$, we can consider  the representation
$ \sigma_n : \G \to G$ given
by $$ \sigma_n(a)= a\quad\text{and}\quad  \sigma_n(b)= b_n .$$
Then  $\sigma_n \to \i_\Ga$ in $\Hom(\Ga, G)$ as $n\to \infty$.
Let $\theta=\{\alpha_1, \alpha_3\}$.
Since $\Ga$ is $\theta$-Anosov by Lemma \ref{ta} and
the set of $\theta$-Anosov representations is open in $\Hom(\Ga, G)$ \cite{GW}, each
$\sigma_n(\Ga)$ is a $\theta$-Anosov subgroup, and hence a discrete subgroup of $G$ for all sufficiently large $n$.

 Since $\sigma_n(\Ga)$ is Zariski dense in $G$ (as $b_n\notin \cal Z$),
the limit cone $\L_{\sigma_n(\G)}$ must be a convex cone of $\fa^+$.
We claim that the upper-semicontinuity of the limit cones fail: we can find an open cone  $\cal C\subset \fa^+$ containing $\L_\Ga-\{0\}$ which does not contain $\L_{\sigma_n(\Ga)}-\{0\}$ for any large enough $n$.

Consider two open cones $\C_1$ and $\C_2$ in $\fa^+$ containing $V_1-\{0\}$ and $V_2-\{0\}$ respectively and $\C_1\cap \C_2 =V_1\cap V_2-\{0\}=V_0-\{0\}$.  By taking $\C_1$ and $\C_2$ small enough,
we can assume that for any unit vectors $w_1\in \C_1-V_0$ and $w_2\in\C_2-V_0$,
the convex combination $\frac{1}{2} (w_1+w_2)$ does not belong to $\C_1\cup \C_2$.
Set $\C=\C_1\cup \C_2$. Since 
$$\L_\Ga\cap V_i\not\subset V_0 $$
by Lemma \ref{ta},
there must be a sequence of unit vectors $u_n\in \L_{\sigma_n(\Ga)}$ converging to some unit vector
$ u\in V_2-V_0$. By the choice of $a$, we have $\mu(a)=(v_1, v_2, 0, v_3)\in V_1-V_0$.
Since $\mu(a)\in V_1-V_0\subset \C_1-V_0$ and $u_n\in V_2-V_0\subset \C_2-V_0$ for all sufficiently large $n$,
we have  $$w_n\coloneqq\frac{1}{2}(u_n +\|\mu(a)\|^{-1}\mu(a))\not\in \C$$ for all large $n$. On the other hand, $w_n$ must belong to in $\L_{\sigma_n(\Ga)}$ by the  convexity of $\L_{\sigma_n(\Ga)}$. 
This shows that  $\L_{\sigma_n(\Ga)}-\{0\} \not\subset \cal C$ for all large enough $n$. 
This finishes the proof. \end{proof}

Note that $\Ga$ is not Zariski dense in $\SL(4, \br)$ in the above example. We remark that
Danciger-Guéritaud-Kassel announced that there exists a Zariski dense Anosov subgroup of $\PSL(2,\br)\times \PSL(2,\br)\times \PSL(2,\br)$ where the limit cone does not vary continuously.

\section{Continuity of \texorpdfstring{$\theta$-}-limit cones}
As before, let $G$ be a connected semisimple real algebraic group.
We fix a non-empty subset $\theta\subset \Pi$.
In this section, we  prove the continuity of $\theta$-limit cones $\L_{\G}^{\theta} $ of $\theta$-Anosov subgroups $\G$ of $G$ by a similar argument as the proof of Theorem \ref{m1}.
 
\begin{theorem}\label{thm:limitcone2}\label{tl}
    Let $\G$ be a $\theta$-Anosov subgroup of $G$ such that
   $\L_{\G}^{\theta}$ is convex (e.g., Zariski dense).
     Then  the map
    \[
     \sigma \mapsto \L_{\sigma(\G)}^{\theta}\in \mathsf C(\fa_\theta^+)
    \]
    is continuous at $\id_\Ga\in \Hom(\Ga, G)$.
    \end{theorem}
Since the projection $p_\theta$ is $\W_\theta$-invariant, we have that $p_\theta^{-1}(\L_\Ga^\theta)$ is a $\theta$-convex cone for $\Ga$ as above, and this is why we do not require the $\theta$-convexity in Theorem \ref{thm:limitcone2}.

\begin{remark}\label{sss} When $\Ga$ is a
fundamental group of a closed negative curved manifold, embedded in $G$ as a Zariski dense $\theta$-Anosov subgroup, Sambarino \cite{Sam1} proved this theorem  using thermodynamic formalism.
Given the work of Bridgeman, Canary, Labourie and Sambarino \cite{BCLS} which provides the thermodynamic formalism for a general $\theta$-Anosov subgroup, his argument should extend to a general Zariski dense $\theta$-Anosov subgroup.

The purpose of this section is to present an entirely different proof of this result using Proposition \ref{ucp}. We believe that this geometric proof is more intuitive and has the potential to extend to a more general setup.
\end{remark}

As before, we need to know the existence of $\theta$-admissible cones provided by the following lemma:
\begin{lem}\label{theta}
  Let $\cal D \subset \fa^+$ be  a closed $\i$-invariant cone such that $p_\theta(\cal D)$ is convex
  and $\cal D\cap \ker \alpha=\{0\}$ for all $\alpha\in \theta$.
  Then for any open cone $\C_0$ in $\fa^+_\theta$ containing $p_\theta(\cal D) - \{0\}$, there exists a $(\theta\cup\i(\theta))$-admissible cone $\cal C$ in $\fa^+$ such that
  $$ \cal D -\{0\}\subset \inte\C \quad\text{ and}\quad p_\theta(\cal C) \subset \C_0 \cup\{0\}.$$
\end{lem}

\begin{proof} 
Since 
$p_\theta(\cal D) $ is convex,
we may assume without loss of generality that $\cal C_0$ is convex.
Moreover, since $\cal D$ is $\i$-invariant, it follows that $p_{\theta\cup\i(\theta)}(\cal D)$ is $\i$-invariant.

Consider the $\i$-invariant open convex cone in $\fa_{\theta\cup\i(\theta)}$:
\[\C_1 \coloneqq p_\theta^{-1}(\C_0) \cap \i (p_\theta^{-1}(\C_0)) \cap \fa^+_{\theta\cup\i(\theta)}.\]  Since $p_{\theta\cup\i(\theta)}(\cal D)$ is $\i$-invariant, $\C_1$ contains $p_{\theta\cup\i(\theta)}(\cal D)-\{0\}$.

Let $\alpha\in \theta\cup\i(\theta)$. As in the proof of Lemma \ref{con1},
 both $\inte \fa_{\theta\cup\i(\theta)}^+$ and  $\cal W_\theta {} \cal D -\{0\}$ are contained in the open half-space $\{\alpha>0\}$. 
Since $\C_1$ is an open cone contained in $\fa_{\theta\cup\i(\theta)}^+$ and hence $\C_1\cap \ker\alpha=\emptyset$,
we can find a linear form $h_\alpha\in \fa^*$ such that
$$(\ker \alpha\cap \fa^+ ) -\{0\} \subset \{h_\alpha<0\} \quad\text{and}\quad 
(\cal C_1\cup  \cal W_{\theta \cup \i(\theta)} \cal D) -\{0\}\subset 
\{h_\alpha >0\} .$$

Now set $H\coloneqq\bigcap_{\alpha\in {\theta\cup\i(\theta)}, w\in \cal W_{\theta\cup\i(\theta)}} \{h_\alpha\circ \op{Ad}_{w} \ge 0\}$, which is clearly a $\cal W_{\theta\cup\i(\theta)}$-invariant convex cone. By our choice of $h_\alpha$, the interior
$\inte H$ contains $\cal D-\{0\}$. We have that $\i(H)$ is also a $\cal W_{\theta\cup\i(\theta)}$-invariant convex cone whose interior contains $\cal D-\{0\}=\i(\cal D)-\{0\}$.

Define \[\C\coloneqq 
p_{\theta\cup\i(\theta)}^{-1}(\bar\C_1) \cap \fa^+\cap  H \cap \i(H).\]
Then $\C$ is a closed cone  whose interior contains $\cal D-\{0\}$, disjoint from
$\ker\alpha$ for all $\alpha\in \theta\cup\i(\theta)$,
 $p_\theta(\cal C)\subset \C_0\cup\{0\}$,
 and $\i(\cal C)=\cal C$.
 Since $p_{\theta\cup\i(\theta)} : \fa \to \fa_{\theta\cup\i(\theta)}$ is $\W_{\theta\cup\i(\theta)}$-invariant,  the preimage
$p_{\theta\cup\i(\theta)}^{-1}(\C_1)$  is a $\W_{\theta\cup\i(\theta)}$-invariant convex cone of $\fa^+$.
Hence we have $$\cal W_{\theta\cup\i(\theta)} \C  = p_{\theta\cup\i(\theta)}^{-1}(\bar\C_1) \cap (\cal W_{\theta\cup\i(\theta)} \fa^+)\cap H\cap \i(H);$$ 
    so $\cal W_{\theta\cup\i(\theta)}  \C $ is convex, being the intersection of convex subsets.
    Therefore $\C$ is $({\theta\cup\i(\theta)})$-admissible.   
\end{proof}

\begin{proof}[Proof of Theorem \ref{thm:limitcone2}]
The lower semicontinuity follows from \Cref{LCA}.
To prove the upper semicontinuity, let $\C_0$ be any open cone in $\fa^+_\theta$ containing $\L_{\G}^\theta-\{0\}$. 

By applying Lemma \ref{theta} with $\cal D=\L_\Ga$,
we can choose a $({\theta\cup\i(\theta)})$-admissible closed cone
 $\C\subset \fa^+$
    such that $$\L_\Ga-\{0\}\subset \inte \C\quad \text{ and}\quad 
     \C\subset  p_\theta^{-1} (\C_0\cup\{0\}).$$ 
By Proposition \ref{ucp},
for all sufficiently large $n$,
$$\L_{\sigma_n(\G)}  \subset p_\theta^{-1} (\C_0\cup\{0\}),$$
and hence
$\L_{\sigma_n(\Ga)}^{\theta }  \subset \cal C_0 \cup\{0\}$.
This implies the upper semicontinuity.
\end{proof}

\section{\texorpdfstring{$\theta$-}-growth indicators vary continuously}
Let $G$ be a connected semisimple real algebraic group.  We fix a non-empty subset $\theta\subset \Pi$.
In this section, we show that the $\theta$-growth indicators vary continuously  in the subspace of Zariski dense $\theta$-Anosov representations of $\Hom(\Ga, G)$.
We consider the $\theta$-Cartan projection:
$$\mu_\theta=p_\theta\circ \mu:G\to \fa_\theta^+.$$

\begin{Def}[$\theta$-growth indicator] \label{growth}  Let $\Ga$ be a $\theta$-discrete subgroup of $G$, i.e.,
$\mu_\theta|_{\Ga}$ is a proper map.
The  $\theta$-growth indicator $\psi_\Ga^{\theta}:\fa_\theta^+\to \{-\infty\}\cup \br   $ is defined  as follows: if $u \in \fa^+_\theta$ is non-zero,
\be\label{gi2} \psi_\Ga^{\theta}(u)=\|u\| \inf_{u\in \cal C}
\tau^\theta_{\mathcal C}\ee 
where $\cal C\subset \fa^+_\theta$ ranges over all open cones containing $u$, and $\psi_{\Ga}^{\theta}(0) = 0$.
Here $-\infty\le \tau^{\theta}_{\cal C}\le \infty$ denotes the abscissa of convergence of the series ${\mathcal P}^{\theta}_{\cal C}(s)=\sum_{\ga\in \Ga, \mu_\theta(\ga)\in \mathcal C} e^{-s\|\mu_\theta(\ga)\|}$.
\end{Def} This definition is equivalent to the one given in \eqref{ttt0},
as was studied in (\cite{Q2}, \cite{KOW}).
 The $\theta$-discreteness hypothesis is necessary for the $\theta$-growth indicator to be well-defined. It follows from Definition \ref{anosov} that $\theta$-Anosov subgroups are $\theta$-discrete.
For $\theta=\Pi$, we simply write $\psi_\Ga=\psi_\Ga^{\Pi}$.
For any discrete subgroup $\Ga<G$,
we have
\begin{equation}\label{rrr}
    \psi_\Ga\le 2\rho
\end{equation}
where $\rho=\frac{1}{2} \sum_{\alpha\in \Phi^+} m(\alpha) \alpha$ is the half sum of all positive roots of $G$ counted with multiplicity (cf. \cite{Q2}).

The main goal of this section is to prove the following theorem:
\begin{theorem}\label{thm:uc}\label{uc} Let $\G$ be a Zariski dense $\theta$-Anosov subgroup of $G$.
Let $\sigma_n\to \id_\Ga$ in $\Hom(\Ga, G)$. Then for any $v\in \inte \L_\Ga^{\theta}$,
we have
\be\label{eqn0:thm:uc} \lim_{n\to \infty} \psi^\theta_{\sigma_n(\Ga)}(v) = \psi^\theta_\Ga (v) .\ee
Moreover, the convergence is uniform on compact subsets of $\inte \L_{\Ga}^\theta$.
\end{theorem}

We begin by recalling the following property of the $\theta$-growth indicator:
\begin{lem} (\cite{Q2}, \cite{KOW}) \label{stp} For any Zariski dense $\theta$-discrete subgroup $\Ga$ of $G$,
 $ \psi_\Gamma^{\theta} $  is upper semicontinuous and concave. 
   Moreover, $ \psi_\Gamma^{\theta} $ is strictly positive on $ \operatorname{int} \mathcal{L}_\Gamma^\theta $ and $\L_\Ga^\theta=\{\psi_\Ga^\theta\ge 0\}$.
\end{lem}
 As a consequence of these properties, $ \psi_\Gamma^{\theta} $ is continuous on $ \operatorname{int} \mathcal{L}_\Gamma^\theta $, since it is a {\em real} valued concave function on $\L_\Ga$ and concave functions are continuous on the interior of their domain.

\subsection*{Continuity of $\theta$-limit sets}
Let $P_\theta$
denote a standard parabolic subgroup of $G$ corresponding to $\theta$; that is, $P_\theta$ is
generated by the centralizer of $A$ and all root subgroups $U_\alpha$, $\alpha \in \Phi^+\cup [\Pi-\theta]$ 
 where $[\Pi-\theta]$
denotes the set of all roots in $\Phi$ which are $\Z$-linear combinations of $\Pi-\theta$.
We set $$\F_\theta=G/P_\theta.$$
\begin{Def}[$\theta$-limit set] \label{tl}
For a discrete subgroup $\Ga<G$, the $\theta$-limit set $\La^\theta_\Ga$ is a closed subset of $\F_\theta$
 defined
as the set of all limit points
$\lim k_i P_\theta\in \F_\theta$ where $k_i\in K$ is a sequence such that
there exists $\ga_i=k_i a_i k_i'\in \Ga\cap K\exp \fa^+K$ with
$\alpha(\log a_i)\to \infty$ for all $\alpha\in \theta$.
\end{Def}

This is clearly a $\Ga$-invariant closed subset which may be empty in general. When $\Ga$ is Zariski dense in $G$, $\La^\theta_\Ga$ is the unique $\Ga$-minimal subset of $\cal F_\theta$ (\cite{Be}, \cite[Corollary 5.2, Lemma 6.3, Theorem 7.2]{Q1}, \cite[Lemma 2.13]{LO}).

One important feature of a $\theta$-Anosov subgroup $\Ga$ is that
$\Ga$ is a hyperbolic group. If $\partial_\infty\Ga$ denotes the Gromov boundary of $\Ga$,
then  there exists a $\Ga$-equivariant continuous injection $f:\partial_\infty \Ga\to \F_\theta$ such that for $\xi\ne \eta$ in $\partial_\infty\Ga$, $f(\xi)$ and $f(\eta)$ are antipodal, that is, the pair $(f(\xi), f(\eta))$ belongs to the unique open $G$-orbit in $G/P_\theta\times G/P_\theta$ under the diagonal action. Moreover the image $f(\partial \Ga)$ coincides with the $\theta$-limit set  $\La^\theta_\Ga$.

Let $\mathsf C(\F_\theta)$ be the space of all closed subsets of $\F_\theta$ equipped with the topology of Hausdorff convergence: $F_i\to F$ if and only if for all $\epsilon>0$, there exists $n\in\N$ such that for all $i\ge n$, $F$ lies in the $\epsilon$-neighborhood of $F_i$ and vice versa.
\begin{theorem}\cite[Theorem 5.13]{GW} 
\label{lc}  Let $\Ga$ be a $\theta$-Anosov subgroup of $G$.
If $\sigma_n\to \id_\Ga$ in $\Hom(\Ga, G)$,
then $\La_{\sigma_n(\Ga)}^\theta$ converges to $\La_\Ga^\theta$ in $\mathsf C(\F_\theta)$.
\end{theorem}

\subsection*{Conformal measures on limit sets} 
The $\frak a$-valued Busemann map $\beta: \cal F\times G \times G \to\frak a $ is defined as follows: for $\xi\in \cal F$ and $g, h\in G$,
$$  \beta_\xi ( g, h):=\sigma (g^{-1}, \xi)-\sigma(h^{-1}, \xi)$$
where  $\sigma(g^{-1},\xi)\in \fa$ 
is the unique element such that we have the Iwasawa decomposition $g^{-1}k \in K \exp (\sigma(g^{-1}, \xi)) N$ for any $k\in K$ with $\xi=kP$ (see \cite{Q1} and \cite[Definition 3.1]{LO}).

The space $\fa_\theta^*$ of all linear forms on $\fa_\theta$ can be identified with the following space:
$$\fa_\theta^*=\{\phi\in \fa^*: \phi=\phi \circ p_\theta\} .$$
In particular, we may regard a linear form on $\fa_\theta$ as a linear form on the whole space $\fa$.
\begin{Def} For $\phi\in \fa_\theta^*$, a probability measure $\nu$ on $\F_\theta$ is called a $(\Ga, \phi)$-conformal measure if
for all $\ga\in \Ga$ and $\xi\in \F_\theta$,
\be\label{conf} \frac{d\ga_*\nu}{d\nu}(\xi)= e^{\phi(\beta_\xi (o, \ga o))} .\ee 
\end{Def}

A linear form $\phi\in \fa_\theta^*$ is said to be  tangent to $\psi_\Ga^\theta$ at a non-zero vector $v\in \fa_\theta^+$ if
$$\text{ $\phi\ge \psi_\Ga^\theta\quad $ and $\quad \phi(v)=\psi_\Ga^\theta(v)$.}$$

The following gives a complete classification of conformal measures supported on the $\theta$-limit sets:
\begin{theorem}\label{larger} (\cite{LO}, \cite{Sam}, \cite[Corollary 1.13]{KOW}) Let $\Ga$ be a Zariski dense $\theta$-Anosov subgroup of $G$.
 For any unit vector $v\in \inte\L^\theta_\Ga$, there exists a unique  tangent linear form $\phi_v\in \fa_\theta^*$ to $\psi_{\Ga}^{\theta}$ at $v$ and 
 there exists a unique $(\Ga, \phi_v)$-conformal measure $\nu_v$ on $\La_\Ga^{\theta}$. 
 
 Moreover, the assignment $v\mapsto \phi_v \mapsto  \nu_v$ gives bijections among 
 the set of all unit vectors in $\inte\L^\theta_\Ga$, the set of all tangent linear forms to
 $\psi_\Ga^\theta$, and the set of all $\Ga$-conformal measures supported on $\La_\Ga^\theta$.
\end{theorem}

\subsection*{Continuous variations of growth indicators} If $\phi\in \fa_\theta^*$ is positive on $\L_\Ga^\theta-\{0\}$,
the critical exponent
$0<\delta_{\phi, \Ga}<\infty$, which is the abscissa of convergence of the series
$s\mapsto \sum_{\ga\in \Ga} e^{-s\phi(\mu(\ga))}$, is well-defined \cite[Lemma 4.2]{KOW} and is equal to
\be\label{ct} \delta_{\phi, \Ga} \coloneqq\limsup_{T\to \infty}\frac{1}{T}{\log \#\{\ga\in \Ga: \phi(\mu(\ga))\le T\}}.
\ee 

\begin{theorem} \label{bcls}
    Let $\Ga<G$ be a $\theta$-Anosov subgroup with $\L_\Ga^{\theta}$ convex. Let $\phi \in \fa_{\theta}^*$ be positive on $\L_{\Ga}^{\theta}-\{0\}$. Then for all $\sigma\in \Hom(\Ga, G)$
    sufficiently close to $\id_\Ga$, the critical exponent $0< \delta_{\phi, \sigma(\Ga)}<\infty$ is well-defined
    and the map $$\sigma\mapsto \delta_{\phi, \sigma(\Ga)}$$ is continuous at $\id_\Ga\in \Hom(\Ga, G)$.
    Moreover, if $D$ is an analytic family of $\theta$-Anosov representations in
    $\Hom(\Ga, G)$, then $\sigma\mapsto \delta_{\phi, \sigma(\Ga)}$ is analytic in $D$.
\end{theorem}

\begin{proof} In \cite[Proposition 8.1]{BCLS} (see also \cite[Section 4.4]{CS}), this is obtained for any $\phi\in \fa_\theta^*$ which is positive on $\fa_\theta^+-\{0\}$, but the condition $\phi>0$ on $\fa_\theta^+-\{0\}$ was needed only to guarantee that $\phi$ is positive on $\L_{\sigma (\G)}^{\theta} $ for all $\sigma$ sufficiently close to $\id_\Ga$.  By Theorem \ref{thm:limitcone2}, if $\phi>0$ on $\L_{\Ga}^{\theta}-\{0\}$, then
$\phi>0$ on $\L_{\sigma(\Ga)}^{\theta} -\{0\}$ for all $\sigma$
sufficiently close to $\id_\Ga$ and therefore the above formulation holds.
\end{proof}

We now give a proof of Theorem \ref{uc}; the use of conformal measures in this context was inspired by the work of Sullivan \cite{Su} and McMullen \cite{Mc} where conformal measures were used to study critical exponents of Kleinian groups.

\medskip 

 \noindent{\bf Proof of Theorem \ref{uc}}
 Let $v\in \inte\L_\Ga^{\theta}$ be a unit vector. By Theorem \ref{larger},
 there exists a unique linear form $\phi=\phi_v\in \fa_\theta^*$ tangent to $\psi_\Ga^\theta$ at $v$.
The $\theta$-Anosov property of $\Ga$ implies that $\psi_\Ga^\theta$ is vertically tangent, i.e., there is no linear form tangent to $\psi_\Ga^\theta$ at a non-zero vector $v\in \partial\L_\Ga^\theta$ 
(\cite{Sam}, \cite[Theorem 12.2]{KOW}). This implies that
$\phi>0$ on $\L_{\Ga}^{\theta}-\{0\}$.
By Theorem \ref{bcls}, we have
$0<\delta_{\phi,\sigma_n(\Ga)}<\infty$ for all large $n$. This implies that
$\delta_{\phi,\sigma_n(\Ga)}\phi$ is a tangent linear form \cite[Lemma 4.5]{KOW}, in particular,
 \be\label{u1} \psi^\theta_{\sigma_n(\Ga)}\le \delta_{\phi,\sigma_n(\Ga)}\phi .\ee 
Since $\delta_{\phi,\sigma_n(\Ga)}\to \delta_{\phi,\Ga}$ by Theorem \ref{bcls}
and $\delta_{\phi,\Ga}=1$ \cite[Corollary 4.6]{KOW},
 we have
 \be\label{bdd} \limsup \psi^\theta_{\sigma_n(\Ga)}(v)\le \limsup \delta_{\phi,\sigma_n(\Ga)} \phi(v)=\phi(v)=\psi^\theta_\Ga(v).\ee

By \Cref{co}, $$v\in \inte \L_{\sigma_n(\Ga)}^{\theta}$$
for all large $n$.
Therefore, by Theorem \ref{larger}, there exists a linear form $\phi_n\in \fa_\theta^*$
tangent to  $ \psi_{\sigma_n(\Ga)}^{\theta}$ at $v$
and a $(\sigma_n(\Ga), \phi_n)$-conformal probability measure, say $\nu_n$, on the limit set
$\La_{\sigma_n(\Ga)}^\theta\subset \F_{\theta}$.

We claim that by passing to a subsequence, $\nu_n$ weakly converges to a $\G$-conformal measure on $\La_\Ga^\theta$. To see this,
fix a closed cone $\cal D$ such that $v\in \inte\cal D\subset
\cal D \subset \inte\L_\Ga^\theta\cup\{0\}$. By Corollary \ref{co}, $\cal D\subset \inte\L_{\sigma_n(\Ga)}^\theta\cup\{0\}$ for all large enough $n\ge 1$.
Since $\psi_{\sigma_n(\Ga)}^\theta\ge 0 $ on $\L^\theta_{\sigma_n(\Ga)}$, we have that for all large $n$,
$$\phi_n \ge 0\quad\text{ on $\cal D$.  } $$
Moreover, since $\cal D$ has non-empty interior and $ \phi_n(v) = \psi_{\sigma_n(\Ga)}^{\theta}(v)\le \psi_\Ga^\theta(v)$ for all large $n$ by \eqref{bdd}, it follows that
the sequence
$\phi_n$ converges, up to passing to a subsequence, to
some linear form, say $\phi\in \fa_\theta^*$, which can possibly be the zero linear form at this point.  Let $\nu$ be a weak-limit of the sequence $\nu_n$. Since 
 $\sigma_n \to \id_\Ga$, it follows from \eqref{conf} that
$\nu$ is a $(\Ga, \phi)$-conformal measure  on  $\F_{\theta}$. 
Since $\La_{\sigma_n(\Ga)}^\theta$ converges to $\La_{\Ga}^\theta$ as $n\to \infty$ by Theorem \ref{lc},  $\nu$ is indeed supported on $\La_\Ga^\theta$. This proves the claim.

Therefore by Theorem \ref{larger}, we have $$\phi \ge \psi_\Ga^{\theta};$$
for instance, $\phi$ cannot be zero.
Hence
\be\label{f} \psi_\Ga^{\theta}(v)  \le \phi(v)=\lim \phi_n(v)=\lim \psi^\theta_{\sigma_n(\Ga)}  (v) .\ee 

Together with \eqref{bdd}, this proves 
$$\psi_\Ga^{\theta}(v)  =\lim \psi^\theta_{\sigma_n(\Ga)}  (v). $$
Finally, for any compact subset $Z\subset \inte \L_\G^\theta$,
 we need to show that as $n\to \infty$,
\be\label{uni} \psi^\theta_{\sigma_n(\Ga)}\vert_Z\to \psi^\theta_\Ga\vert_Z \quad\text{uniformly on } Z.
   \ee 

By \Cref{co}, we have $Z\subset \inte \L_{\sigma_n(\G)}^\theta$ 
   for all sufficiently large $n\ge 1$.
   Since $\inte\L^\theta_\Ga$ is a convex subset,
  any compact subset $Z\subset \inte \L^\theta_\Ga$ can be covered by finitely many compact convex subsets of $\inte \L^\theta_\Ga$.
Hence we may assume without loss of generality that $Z$ is a compact convex subset.

Recall that each $\psi^\theta_{\sigma_n(\Gamma)} $ is continuous on $\inte_{\sigma_n(\Ga)}^\theta$ (see the remark following Lemma \ref{stp}), and hence
we may assume that $\psi^\theta_{\sigma_n(\Gamma)}\vert_{Z} $
is continuous for all $n\ge 1$.
We claim that the family $\cal E_Z\coloneqq \{ \psi^\theta_{\sigma_n(\Gamma)}\vert_{Z} : n \in \mathbb{N} \} $ is equicontinuous.
Observe that $\cal E_Z$ consists of  positive concave continuous functions. Moreover, 
by \cite[Lemma 3.13]{KOW}, 
\[\psi_{\sigma_n(\Ga)}^\theta(u) \le \max_{v\in p_\theta^{-1}(u)}\psi_{\sigma_n(\Ga)}(v)\quad\text{for all }u\in\fa^+_\theta.\]
Since $\psi_{\sigma_n(\Ga)}\le 2\rho$ by \eqref{rrr}, we have that
\be\label{GIbound} \max_{u\in Z}\psi_{\sigma_n(\Ga)}^\theta(u)
\le \max_{v\in p_\theta^{-1}(Z)\cap \L_{\sigma_n(\G)}} 2\rho(v).\ee

Consider $\theta\cup\i(\theta)$-admissible closed cones $\cal D\subset \cal C$
such that $\L_{\G}-\{0\}\subset \inte \cal D$ and $\cal D-\{0\} \subset \inte \C$, which exist by
Lemma \ref{theta}. By \Cref{ucp}, we have that $\L_{\sigma_n(\G)} \subset \C$ for all large $n\ge 1$.  Since $\C$ is $\theta\cup\i(\theta)$-admissible and hence $\cal C\cap \ker\alpha=\{0\}$ for all $\alpha\in \theta\cup\i(\theta)$, the intersection $p^{-1}_\theta(Z) \cap \C$ is compact.
Thus the right side of \eqref{GIbound} is uniformly bounded by 
\[
\max_{v\in p_\theta^{-1}(Z)\cap \C} 2\rho(v) <\infty.
\]
Hence, $\cal E_Z$ is a uniformly bounded family of concave functions.
Since this is true for an arbitrary compact convex subset $Z\subset \inte\L_\G^\theta$, it follows that the functions in $\cal E_Z$ are uniformly Lipschitz, which implies that they are equicontinuous.

By the Arzelà--Ascoli theorem, any sequence from the family $ \cal E_Z $ has a subsequence that converges uniformly. Since the growth indicator functions $ \psi^\theta_{\sigma_n(\Gamma)}|_Z $ converge pointwise to $ \psi_\Gamma^\theta|_Z$, any uniform limit of $\cal E_Z$ must be $ \psi_\Gamma^\theta \vert_Z$. 
This proves \eqref{uni}, completing the proof of  Theorem \ref{uc}.

\medskip 

 \noindent{\bf Proof of Corollary \ref{us}:} Let $u_\Ga^\theta\in \fa_\theta$ be
 the unique unit vector defined by
\be\label{eqn:max}
\delta^\theta_\Ga= \psi^\theta_\Ga (u_\Ga^{\theta})=\max_{\|u\|=1, u\in \fa_\theta^+}\psi^\theta_\Ga (u) .
\ee

As $\Ga$ is $\theta$-Anosov, we have that $\psi_\Ga^\theta$ is strictly concave and vertically tangent \cite[Theorem 12.2]{KOW}. It follows that
 $u_{\Ga}^{\theta}\in \inte \L_\Ga^{\theta}. $

Choose a compact neighborhood $B \subset  \inte\L_{\G}^{\theta} \cap \fa^1_\theta$ of $u_{\Ga}^\theta$.
Let $\sigma_n$ be a sequence converging to $\id_\Ga$ in $\Hom(\Ga, G)$.
By \Cref{co}, we have
\[
 B \subset \inte\L_{\sigma_n(\G)}^{\theta} \cap \fa^1_\theta \quad \text {for all sufficiently large }n\ge 1.
\]

Let  $u_n\in B$ be the unit vector such that
$\psi_{\sigma_n(\Ga)}^{\theta}(u_n)$ attains the maximum of $\psi_{\sigma_n(\Ga)}^{\theta}$ on $B$. We claim that
 \be\label{un} u_n = u^\theta_{\sigma_n(\G)}, \ee  that is, 
$\psi_{\sigma_n(\Ga)}^{\theta}(u_n)=\max_{v\in \fa^1_\theta} \psi_{\sigma_n(\Ga)}^\theta(v)$.

By definition of $u_{\Ga}^{\theta}$, there exist $\eta>0$ such that 
 $$\psi_{\Ga}^{\theta} (u_{\Ga}^{\theta}) \ge \eta + \max_{v\in \partial B} \psi_{\Ga}^{\theta}(v) .$$
Since $\psi_{\sigma_n(\Ga)}^{\theta}$ converges uniformly to $ \psi_{\Ga}^{\theta}$ on $B$ by \Cref{thm:uc}, we have, for all sufficiently large enough $n\ge 1$,
 \be\label{ineq} \psi_{\sigma_n(\Ga)}^{\theta} (u_{\Ga}^{ \theta})  \ge \frac{\eta}{2} + \max_{v\in \partial B} \psi_{\sigma_n(\Ga)}^{\theta}(v).\ee 
Since  $\psi_{\sigma_n(\Ga)}^{\theta} (u_n) \ge  \psi_{\sigma_n(\Ga)}^{\theta} (u_{\Ga}^{ \theta})$,  it follows from \eqref{ineq} that $u_n$ lies in
in the interior of $B$ for all large $n \ge 1$. 
Now suppose that  $u_n \ne  u^\theta_{\sigma_n(\G)}$. This means that $u^\theta_{\sigma_n(\G)} \notin B$.
Noting that both $u_n$ and $u^\theta_{\sigma_n(\Ga)}$ lie in $\inte\L^\theta_\Ga$, which is convex,
consider the line segment in $ \inte \L_\Gamma^\theta $ connecting $ u_n $ to $ u^\theta_{\sigma_n(\G)} $:
\[
t \mapsto c(t) \coloneqq t u_n + (1-t) u^\theta_{\sigma_n(\G)} \quad \text{for } t \in [0,1].
\]
Note that $0<\| c(t) \| < 1 $ for all $ t\in (0,1) $.
Therefore, by the concavity and homogeneity of $ \psi_{\sigma_n(\Ga)}^{\theta} $, we have that for all $ t \in (0,1) $,
\begin{align*} 
\psi_{\sigma_n(\Ga)}^{\theta} \left(\frac{c(t)}{\| c(t) \|}\right)
&= \frac{1}{\| c(t) \|} \psi_{\sigma_n(\Ga)}^{\theta} (c(t)) \\
&> \psi_{\sigma_n(\Ga)}^{\theta} (c(t)) \\
&\ge t\psi_{\sigma_n(\Ga)}^{\theta}(u_n) + (1-t) \psi_{\sigma_n(\Ga)}^{\theta}(u^\theta_{\sigma_n(\G)}) \\
&\ge \min \left\{ \psi_{\sigma_n(\Ga)}^{\theta}(u_n), \, \psi_{\sigma_n(\Ga)}^{\theta}(u^\theta_{\sigma_n(\G)}) \right\} \\
&= \psi_{\sigma_n(\Ga)}^{\theta}(u_n).
\end{align*}

Since $\frac{c(t)}{\| c(t) \|}\to u_n$ as $t\to 1$, this is a 
contradiction to the fact that $ u_n $ attains the maximum of $ \psi_{\sigma_n(\Ga)}^{\theta} $ on $B$. Therefore, $u^\theta_{\sigma_n(\G)} \in B$, proving the claim \eqref{un}.

Since $B$ is an arbitrary compact neighborhood of $u_\Ga^\theta$, it follows from \eqref{un} that
\be\label{fin1}
u_n=u^\theta_{\sigma_n(\G)} \to u^\theta_{\G} \quad \text{as } n\to\infty.
\ee 

Since $\psi_{\sigma_n(\Ga)}^\theta\to \psi_\Ga^\theta$ uniformly on $B$,
by  \Cref{thm:uc}, we obtain from \eqref{fin1} that 
$$\delta_{\sigma_n(\Ga)}^{\theta} \to \delta_{\Ga}^{\theta}\quad\text{ as $n\to\infty$.}$$ This finishes the proof.

\section{Continuity of more general critical exponents}
If $\Ga$ is a discrete subgroup of $G$ and $\phi\in \fa^*$ is positive on $\L_\Ga-\{0\}$, then 
the $\phi$-critical exponent 
$\delta_{\phi, \Ga}$ of $\Ga$  is well-defined as in \eqref{ct}.

In this section, we obtain the following continuity of $\phi$-critical exponents as an
application of Theorem \ref{ucp}. The main point of this theorem is that we are not assuming that the linear form $\phi$ is $p_\theta$-invariant unlike in Theorem \ref{bcls}.

\begin{theorem}[\Cref{m2}]\label{m22}
    Let $\Ga$ be a $\theta$-Anosov subgroup of $G$  with $\L_\Ga \subset \fa_\theta^+$ convex. Let $\phi\in \fa^*$ be such that $\phi>0$ on $\L_\Ga -\{0\}$. Then for all $\sigma\in \Hom(\Ga, G)$ sufficiently close to $\id_\Ga$,
    the critical exponent $0<\delta_{\phi, \sigma(\Ga)}<\infty$ is well-defined and 
    the map $\sigma\mapsto 
    \delta_{\phi, \sigma(\Ga)}$ is continuous at  $\id_\Ga \in \op{Hom}(\Ga, G)$.
\end{theorem}

\begin{proof}
Let $\phi\in \fa^*$ be such that $\phi> 0$ on $\L_\Ga-\{0\}$.
Set $$\phi'\coloneqq\phi \circ p_\theta,$$
so that $\phi'\in \fa_\theta^*$. Since
$\phi'=\phi$ on $\fa_\theta$ and $\L_\Ga=\L_\Ga^\theta$, we have
$\phi'> 0$ on $\L_\Ga-\{0\}$. Fix a small open neighborhood $\cal O'$ of $\id_\G$ in $\Hom(\Ga,G)$ so that  $$c\coloneqq\sup_{\sigma\in \cal O'} \delta_{\phi', \sigma(\Ga)} <\infty $$
by Theorem \ref{bcls}.
For $\eta>0$, set $ {\mathcal N}_\eta(\L_\Ga)\coloneqq\{ v\in \fa^+: \|v-\L_\Ga \| \le \eta\|v\|\}.$

Since $\phi>0$ on $\L_\Ga-\{0\}$,
there exists $\eta>0$ such that $\phi$ is positive on
 ${\mathcal N}_\eta(\L_\Ga)-\{0\}$. Hence $\phi$ and the norm
 $\|\cdot \|$ are comparable to each other  on
 ${\mathcal N}_\eta(\L_\Ga)$.
Let $\e>0$ be arbitrary.
By making $\eta$ smaller if necessary, we may assume that
for all $v\in {\mathcal N}_\eta(\L_\Ga)$, we have
\be\label{e2} -\e\phi(v) \le  (\phi - \phi')(v)  \le \e \phi(v).\ee 

We can take a small neighborhood $\cal O\subset \cal O'$ of $\id_\Ga$ in $\Hom(\Ga, G)$ so that for any $\sigma\in \cal O$, $\sigma$ is $\theta$-Anosov and 
\be\label{e1} \L_{\sigma(\Ga)}\subset {\mathcal N}_\eta (\L_\Ga)\ee 
by Theorem \ref{UC}.
In particular, there exists a finite subset $F_\sigma\subset \G$ such that 
$$\mu(\sigma(\Ga-F_\sigma ))\subset  {\mathcal N}_\eta (\L_\Ga).$$

Since  $\lim \delta_{\phi', \sigma(\Ga_0)} = \delta_{\phi', \Ga_0}$ as $\sigma\to \id_\Ga$ by Theorem \ref{bcls}, 
we may assume that for all $\sigma\in \cal O$, we have
\be\label{e3} |\delta_{\phi',
\sigma(\Ga)}-\delta_{\phi', \Ga}|\le \e\ee 
 by replacing $\cal O$  by a smaller neighborhood if necessary. Then for any $\sigma\in \cal O$, we have that for all $s>0$,
$$\sum_{\ga\in \Ga-F_\sigma } e^{-(1-\epsilon) s\phi (\mu(\sigma(\ga)))} \ge \sum_{\ga\in \Ga-F_\sigma} e^{- s\phi'(\mu(\sigma(\ga)))}.  $$
It follows that
$$\delta_{(1-\epsilon) \phi,\sigma(\Ga) } \ge  \delta_{\phi', \sigma(\Ga)} \text{ and hence }
\delta_{\phi,\sigma(\Ga)}  \ge (1-\epsilon)  \delta_{\phi',\sigma(\Ga)}. $$
Similarly, we have that for any $\sigma\in \cal O$,
$$\sum_{\ga\in \Ga-F_\sigma } e^{-(1+\epsilon) s\phi(\mu(\sigma(\ga)))} \le \sum_{\ga\in \Ga-F_\sigma} e^{- s \phi'(\mu(\sigma(\ga)))} . $$
Therefore $$\delta_{(1+\epsilon) \phi,\sigma(\Ga) } \le  \delta_{\phi', \sigma(\Ga)}\text{ and hence }
\delta_{\phi,\sigma(\Ga)}  \le (1+\epsilon)  \delta_{\phi',\sigma(\Ga)}. $$

Therefore for all $\sigma\in \cal O$,
$$ (1-\e) \delta_{\phi', \sigma(\Ga)} \le \delta_{\phi,
\sigma(\Ga)} \le (1+\e) \delta_{\phi',\sigma(\Ga)}.$$

It follows from \eqref{e3} that for all $\sigma\in \cal O$,
	$$ | \delta_{\phi,
\sigma(\Ga)}-\delta_{\phi, \Ga}| \le 2\e \delta_{\phi', \sigma(\Ga)}  \le 2c \e .$$
This finishes the proof.
\end{proof}

\noindent{\bf Proof of Corollary \ref{coo}:} Since $\phi>0$ on $\L_\Ga-\{0\}$, we have that $\phi>0$ on $\L_{\sigma_n(\Ga)}-\{0\}$
for all sufficiently large $n$ by Theorem \ref{UC}.
It follows from \cite{KMO} that $\delta_{\phi, \sigma_n(\Ga)}$ is the minimum constant $\kappa_n\ge 0$ such that
 $$\psi_{\sigma_n(\Ga)}\le \kappa_{n}\phi. $$
Therefore $\kappa \ge \delta_{\phi, \sigma_n(\Ga)}$.
By Theorem \ref{m2}, we have
$$\kappa \ge \delta_{\phi, \Ga} .$$
 This implies that
 $\psi_\Ga \le \kappa \phi$, as desired.

\end{document}